\documentclass[12pt,sort&compress]{elsarticle}
\usepackage{fullpage}
\usepackage{amsmath,amsfonts,amsthm}
\newtheorem{theorem}{Theorem}
\usepackage{graphicx}
\usepackage{float}
\usepackage{hyperref}
\usepackage[titletoc]{appendix}

\newcommand{\ee}{{\mathbf e}} 

\newcommand{\barv}{{\bar v^*}} 
\newcommand{\barvtwo}{{\bar v^{*2}}} 
\newcommand{\om}{{\omega}} 
\newcommand{\cl}{{C_L}}
\newcommand{\clp}{{C_L^{\prime}}}
\newcommand{\clpp}{{C_L^{\prime\prime}}}
\newcommand{\clppp}{{C_L^{\prime\prime\prime}}}
\newcommand{\cd}{{C_D}}
\newcommand{\cdp}{{C_D^{\prime}}}
\newcommand{\cdpp}{{C_D^{\prime\prime}}}

\newcommand{\Lam}{\mbox{\boldmath$\Lambda$}}  

\newcommand{\pa}[2]{\ensuremath{\frac{\partial #1}{\partial #2}}}

\begin{document}
\begin{frontmatter}
	\title{Global phase space structures in a model of passive descent}
	\author{Gary K. Nave, Jr.}
	\ead{gknave@vt.edu}
	\author{Shane D. Ross}
	\address{Engineering Mechanics Program, Virginia Tech \\ Blacksburg, VA, 24061}
	\begin{abstract}
		Even the most simplified models of falling and gliding bodies exhibit rich nonlinear dynamical behavior. Taking a global view of the dynamics of one such model, we find an attracting invariant manifold that acts as the dominant organizing feature of trajectories in velocity space. This attracting manifold captures the final, slowly changing phase of every passive descent, providing a higher-dimensional analogue to the concept of terminal velocity, the terminal velocity manifold.
        Within the terminal velocity manifold in extended phase space, there is an equilibrium submanifold with equilibria of alternating stability type, with different stability basins.
		In this work, we present theoretical and numerical methods for approximating the terminal velocity manifold and discuss ways to approximate falling and gliding motion in terms of these underlying phase space structures.
	\end{abstract}
	\date{Received: 13 April 2018 / Accepted: 18 April 2019}
\end{frontmatter}
\section{Introduction}
A wide variety of natural and engineered systems rely on aerodynamic forces for locomotion. Arboreal animals use gliding flight to catch prey or escape predators \cite{socha2015animals}, while plant seeds may slowly follow the breeze to increase dispersion \cite{minami_various_2003}. To compare different gliding animals with different morphologies, a variety of studies have resolved detailed motion of animals' glides from videos or other tracking methods \cite{jackson2000glide,vernes2001gliding,mcguire2005cost,socha_non-equilibrium_2010,bahlman2013glide}. Throughout their glide, animals may approach an equilibrium glide, but typically spend at least half of the glide between the initial ballistic descent, where aerodynamic forces are small, and the final equilibrium state when aerodynamic forces completely balance weight \cite{mcguire2005cost,socha_non-equilibrium_2010,bahlman2013glide}. Falling seeds exhibit a variety of different falling motions as well, as fluid forces may drive rotation, fluttering, or tumbling as they descent \cite{minami_various_2003}. These behaviors have provided inspiration for physical research and bio-inspired engineering design \cite{vogel1994life,kovavc2011epfl,dickson2013design,leonard2010coordinated}. To compare these behaviors and set the stage for future engineered models, it is useful to consider a mathematical model for this motion.

Both gliding animal flight and plant seed descent represent special cases of passive aerodynamic descent, a topic with very rich history dating back to at least the 19th century in the work of Maxwell and Zhukovskii \cite{maxwell1854particular,lamb1932hydrodynamics,andronov2013theory}. Prior experimental studies have identified several canonical behaviors exhibited by falling disks, plates, and plant seeds characterized by different couplings between rotation and translation in space \cite{minami_various_2003,field_chaotic_1997}. In all of these behaviors, an inherently high-dimensional system, which must consider velocities, angles, and angular velocities in 3D space, converges to a low-dimensional behavior, whether traveling in a single plane or going through a cycle of velocities \cite{field_chaotic_1997,ern2012wake,andersen_unsteady_2005,andersen_analysis_2005,tam2010tumbling,vincent2016holes}. 

Mathematical models have offered many insights into passive descent. Ideal flow theory has been used to study the motion of a body both through a steady fluid and interacting with shed vortices \cite{lamb1932hydrodynamics,aref1993chaotic,borisov2007dynamic,Roenby1871}. The Zhukovskii problem or phugoid model, which assumes that the wing travels with constant angle-of-attack, is a two-dimensional ordinary differential equation for flight from which phugoid oscillations, which couple forward velocity with pitch angle, arise \cite{andronov2013theory}.  Andersen et al. developed a phenomenological model based on experiments and simulations which produce fluttering, tumbling, and even chaotic behavior through a 4D differential equation \cite{andersen_unsteady_2005,andersen_analysis_2005}. To focus specifically on gliding animal flight and compare the gliding capabilities of different animals, Yeaton et al. \cite{yeaton_global_2017} introduced a two-dimensional model for non-equilibrium gliding of animals. It is a modification of this model which we will consider in the present work.
In this model, the authors decoupled translational and rotational dynamics in order to take a deeper view of the translational behavior and shape dependence based on lift and drag characteristics alone. To do so, inspired by the motion of gliding animals, the authors treat pitch angle as fixed with respect to the ground, assuming that the glider has some amount of control to hold this angle. Lift and drag are treated as functional parameters in this model, which is used to capture the differences in glider shapes. This model works especially well for gliding animals, but may be extended to general passive descent.  Yeaton et al. \cite{yeaton_global_2017} analyzed a variety of animal gliders and found that in most examples, trajectories in velocity space collapse onto a single curve as their velocities evolve slowly toward an equilibrium glide, as can be seen in Figure \ref{fig:TVMdemo}. In such examples, most of the dynamics of a passively descending body with constant pitch lie on or near an attracting normally hyperbolic invariant manifold in velocity space.

Attracting normally hyperbolic invariant manifolds, or simply attracting manifolds, such as the one found in Yeaton et al. \cite{yeaton_global_2017}, can play a very useful role in understanding the dynamics of a system \cite{wiggins2013normally}. As most of the dynamics occur near the manifold itself, the dimension of the system may be reduced by projecting the dynamics onto the attracting manifold \cite{kevrekidis2003equation}. In general, attracting manifolds also represent barriers to transport in the evolution of a system and therefore play a key role in understanding, for example, mixing in fluid systems \cite{shadden2005definition,haller_variational_2011,haller_lagrangian_2015}. The attracting manifold observed by Yeaton et al. \cite{yeaton_global_2017} bears similarities to the idea of terminal velocity, acting as a barrier between trajectories which start with small velocity and those which begin with very large velocity. Therefore, we will refer to this curve as the \textit{terminal velocity manifold}, or TVM. This comparison is made more clear in Section \ref{ss:terminalVelocity}. The understanding of such a structure in this model will lead to a clear way to compare a variety of gliders with the same set of tools, will allow for the dimension reduction of the system to understand the structure of these models, and may lead to the possibility of controlling gliding flight with the intuition of the system's global structure.

The purpose of this paper is to investigate the glider model introduced by Yeaton et al. \cite{yeaton_global_2017} more deeply by numerically identifying and analyzing the properties of the terminal velocity manifold. To analyze this manifold, we will use three example gliders, each one representing a potential application for this model. The first is the simplified mathematical model for a falling flat plate \cite{wang2004unsteady}. The second is a biologically-inspired airfoil, based on the cross-sectional shape of a flying snake, \textit{Chrysopelea paradisi} \cite{holden_aerodynamics_2014}. The third represents an engineering application, and uses wind-tunnel measurements on a NACA-0012 airfoil \cite{sheldahl1981aerodynamic}. With these three examples, we will investigate the terminal velocity manifold, methods to find it, and how it changes with pitch angle, both as a bifurcation parameter and later as a time-varying parameter.

We begin the study by introducing the model from Yeaton et al. \cite{yeaton_global_2017} which will be used in our analysis and describing the lift and drag functions for the flat plate, flying snake, and NACA-0012 airfoils in Section \ref{s:GliderModel}. In Section \ref{s:equilibria}, we look at properties of the equilibrium points and bifurcations of the model. In Section \ref{s:TVMdetection}, we discuss the terminal velocity manifold as it relates to stable and unstable manifolds of fixed points and numerical schemes to calculate it by bisection and the trajectory-normal repulsion rate. Once calculated, we vary the pitch parameter to look at changes in the terminal velocity manifold and simulate controlled changes in pitch in Section \ref{s:3DTVM}. From the work of this study, we gain new insights into gliding flight and the structure of slow-fast systems, discussed in Section \ref{s:conclusion}.

\section{A simple glider model}\label{s:GliderModel} 
In their investigation of non-equilibrium gliding flight of animals, Yeaton et al. \cite{yeaton_global_2017} introduce a general model for gliding flight which treats lift and drag as functions of angle-of-attack and pitch as a fixed parameter. This model provides a framework for comparing different gliding animals or objects through a non-dimensional scaling parameter, $\epsilon$. As a dynamical system, the system presents a fascinating model for investigation: it is naturally nonlinear, relies on intuitive assumptions which may be relaxed, and depends on functional input parameters.

\begin{figure}[t]
	\centering
	\includegraphics[width=3in]{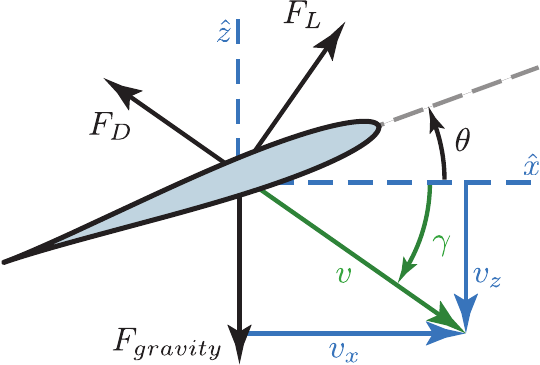}
	\caption{Definitions of angles and directions in the glider model used in this paper. The pitch angle $\theta$ represents the body's fixed orientation with respect to the ground. The black arrows show the lift force $F_L$, drag force $F_D$ and gravitational force $F_{gravity}$ which give the model in (\ref{eq:dimensionalEOM}). The magnitude $v$ and direction $\gamma$ of the body's velocity, shown in green, form the velocity-polar coordinates used in (\ref{eq:glider model}), while the $\hat{x}$ and $\hat{z}$ directions shown in blue comprise the inertial coordinates used in this study.}
	\label{fig:AngleDef}
\end{figure}

To begin, consider a body moving in an unbounded, quiescent fluid medium under the force of gravity. The forces on the body consist solely of gravity and forces which arise from interactions with the surrounding fluid.  The complexity of this problem, then, depends on the choice of model for the fluid forces. To fully capture the physics, a coupled infinite-dimensional fluid-structure model would be required.

For the purposes of presenting a low-dimensional model, however, we will consider only quasi-steady lift and drag, dependent solely on angle of attack, neglecting unsteady fluid forcing and Reynolds number dependence. As shown in the work of Andersen et al. \cite{andersen_unsteady_2005,andersen_analysis_2005}, a quasi-steady assumption for lift and drag captures the dominant behavior in most situations of passive descent. The fluid forces, then, are given by $F_L = \frac{1}{2}\rho S V^2 C_L $ and $F_D = \frac{1}{2}\rho SV^2 C_D$, where $\rho$ represents the density of the surrounding fluid, $S$ represents the spanwise cross-sectional area of the body, $V$ represents the magnitude of the body's velocity, and $C_L$ and $C_D$ represent the projection of the total fluid force onto coordinates perpendicular to (lift) and parallel to (drag) the direction of motion, respectively. The directions of these fluid forces are shown in Figure \ref{fig:AngleDef}. For a glider consisting of an extruded two-dimensional shape, $S=cs$ with $c$ as the chord length of the body and $s$ as the length of the span. As we are neglecting boundary effects and the fluid forces are the same anywhere in physical space, the position of the body represents two ignorable coordinates.

Under these assumptions, the equations of motion are given by the following equations, taking the parallel and perpendicular accelerations to be $a_\parallel=dV/dT$ and $a_\perp=V\,d\gamma/dT$, and the gravitational force as $F_{gravity}=mg$,
\begin{equation}
\begin{aligned}
m\frac{dV}{dT} &= -\frac{1}{2}\rho cs V^2 C_D+mg\sin\gamma, \\
mV\frac{d\gamma}{dT} &= -\frac{1}{2}\rho cs V^2 C_L+mg\cos\gamma,
\end{aligned}
\label{eq:dimensionalEOM}
\end{equation}
where $m$ is the mass of the glider, $g$ represents gravitational acceleration, and $\gamma$ is the clockwise direction of velocity with respect to the horizontal as shown in Figure \ref{fig:AngleDef}.

As in  \cite{yeaton_global_2017}, we introduce nondimensional time $t$ and velocity $v$, choosing $\frac{d}{dT} = \sqrt{g/c\epsilon}\frac{d}{dt}$ and $V =\sqrt{gc/\epsilon}v$ to rescale (\ref{eq:dimensionalEOM}). This rescaling features a nondimensional factor, $\epsilon=\frac{\rho c S}{2m}$, which is the universal glide scaling parameter. It can be used to compare various gliders against one another \cite{yeaton_global_2017}. The dimensionless equations of motion become,
\begin{equation}
\begin{aligned}
\dot{v} & = -C_D v^2 + \sin\gamma, \\
v\dot{\gamma} & = -C_L v^2 + \cos\gamma,
\end{aligned}
\label{eq:dimensionlessEOM}
\end{equation}
where the dot represents differentiation with respect to the non-dimensional time, $\dot{(\:)}=d/dt$.

We consider these fluid force coefficients as functions of angle of attack only. As shown schematically in Figure \ref{fig:AngleDef}, the angle of attack is given by the sum of the constant pitch angle $\theta$, which is the counter-clockwise angle of the body with respect to the ground, and the glide angle $\gamma$, the clockwise angle of the body's velocity with respect to the ground.
\begin{equation}
	\begin{aligned}
		\dot{v} & = -C_D(\gamma+\theta) v^2 + \sin\gamma, \\
		\dot{\gamma} & = -C_L(\gamma+\theta) v + \frac{1}{v}\cos\gamma.
	\end{aligned}
	\label{eq:glider model}
\end{equation}
We will discuss these functions in more detail in Section \ref{ss:liftdrag}.

Alternatively, we can express this system in an inertial reference frame aligned with an observer on the ground, where $v_x=v\cos\gamma$ is the horizontal velocity and $v_z=-v\sin\gamma$ is the vertical velocity. In terms of these inertial coordinates, the total velocity is given by $v=\sqrt{v_x^2+v_z^2}$ and the glide angle is $\gamma=-\arctan{v_z/v_x}$. Allowing the functional dependence of lift and drag to again be implicit, the equations are given by,
\begin{equation}
	\begin{aligned}
		\dot{v}_x &= v^2\left(C_L\left(\gamma+\theta\right)\sin\gamma - C_D\left(\gamma+\theta\right)\cos\gamma\right), \\
		\dot{v}_z &= v^2\left(C_L\left(\gamma+\theta\right)\cos\gamma + C_D\left(\gamma+\theta\right)\sin\gamma\right) - 1.
	\end{aligned}
	\label{eq:inertial-v1}
\end{equation}
These equations of motion end up being the most convenient to use, where we observe the system in inertial coordinates, but calculate the right hand side in terms of the velocity-polar coordinates.

In terms of these inertial coordinates, this system can also be expressed as,
\begin{equation}
	\begin{aligned}
		\dot{v}_x &= \sqrt{v_x^2+v_z^2}\left(-C_L\left(v_x,v_z,\theta\right)v_z - C_D\left(v_x,v_z,\theta\right)v_x\right), \\
		\dot{v}_z &= \sqrt{v_x^2+v_z^2}\left(~~C_L\left(v_x,v_z,\theta\right)v_x - C_D\left(v_x,v_z,\theta\right)v_z\right) - 1.
	\end{aligned}
	\label{eq:inertial-v2}
\end{equation}


\subsection{The terminal velocity manifold}\label{ss:terminalVelocity}
Terminal velocity is a common notion in popular explanations of fluid forces, defined as the value of velocity which balances wind resistance and gravity such that a falling body can no longer accelerate. We can express this concept mathematically with a one-dimensional model of vertical descent, using the same rescalings considered above in (\ref{eq:dimensionlessEOM}),
\begin{figure}[t]
	\centering
	\includegraphics[width=6.5in]{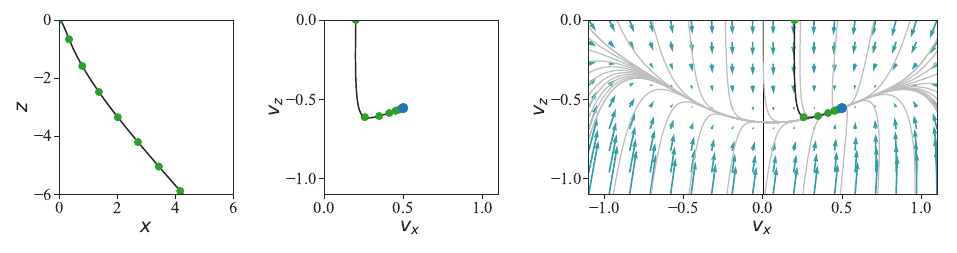}
	\caption{Consider an example glide in position space (left) in which the glider launches with an initial velocity of $\left(v_x(0), v_z(0)\right)=(0.2, 0)$ and the associated velocities of that glide (center). The green points represent even the states at every 1.5 non-dimensional time units. Between the first two green points, the glider accelerates downward due to gravity. After that, the velocities change more slowly toward the equilibrium velocity shown in blue. Looking at the whole velocity space (right), trajectories shown in gray from a variety of initial conditions seem to collapse to a single curve in the velocity space. This attracting curve is the \textit{terminal velocity manifold}. Additionally, the blue vector field shows the magnitude and direction of acceleration at every velocity, while the dark gray vertical line represents $v_x=0$.}
	\label{fig:TVMdemo}
\end{figure}
\begin{equation}\label{eq:TerminalVelocity}
	\dot{v}_z = C_Dv_z^2 - 1.
\end{equation}
This model is identical to (\ref{eq:inertial-v2}) when $v_x=0$, $C_L=0$, and $v_z<0$, because $\sqrt{v^2_z}v_z=-v^2_z$. From this one degree-of-freedom model, terminal velocity is the point where wind resistance (drag) balances gravity, which is the fixed point (\ref{eq:TerminalVelocity}), $v_{T} = -\sqrt{1/C_D}$ in our rescaled coordinates. The terminal velocity point is a zero-dimensional object, serving as a codimension one structure in the one-dimensional model. It acts as a barrier to transport, because all trajectories of velocity approach the fixed point without crossing it. A small initial magnitude of velocity can never become larger than terminal velocity, while a large initial magnitude of velocity can never become smaller than terminal velocity. It divides the one-dimensional phase space into two qualitatively distinct regions; those approaching the terminal velocity from below and those approaching it from above.

As discussed in the introduction, the \textit{terminal velocity manifold}, or TVM, divides phase space into two regions without allowing trajectories to cross it on their way to an ultimate fixed point. The attracting one-dimensional structure in the two-dimensional model of \ref{eq:glider model}, shown in Figure \ref{fig:TVMdemo}, acts as a higher-dimensional analogue of terminal velocity. All trajectories rapidly converge onto the manifold and slowly evolve along it regardless of initial condition, just as all trajectories converge to the point of terminal velocity in the one-dimensional model (\ref{eq:TerminalVelocity}). The difference in time scales onto and along the manifold can be seen by the equally-spaced snapshots in time represented by green dots in Figure \ref{fig:TVMdemo}.


\subsection{Lift and drag as functional parameters}\label{ss:liftdrag}
The behavior of this model depends entirely on the choice of pitch parameter and the lift and drag functions for a given airfoil. For any object without an axisymmetric cross-sectional shape, the fluid forces will depend on the angle-of-attack of the object. As discussed before and shown in Figure \ref{fig:AngleDef}, this angle of attack in our model is the sum of the pitch angle $\theta$ and the glide angle $\gamma$: \(\alpha= \theta+\gamma\). We assume that the fluid force coefficients are independent of increasing Reynolds number.

The lift coefficient is a function which maps from angle of attack, a cyclic variable on $\mathbb{S}^1$, to a finite subset of $\mathbb{R}$, as an unbounded lift coefficient would lead to unbounded acceleration,
\begin{equation}\label{eq:liftCoeff}
C_L:\mathbb{S}^1\mapsto\mathcal{I}_L\subset\mathbb{R}.
\end{equation}
The value may be positive or negative.

The range of the drag function, $\mathcal{I}_D$, on the other hand, must be positive. In a quiescent fluid, a passively falling body cannot produce thrust on its own, and, although small, there must be some small amount of viscous drag on the body,
\begin{equation}\label{eq:dragCoeff}
C_D:\mathbb{S}^1\mapsto\mathcal{I}_D\subset\mathbb{R}^+.
\end{equation}
The details of each function depend on the body's shape. Below, we consider three example systems: a mathematical model for a falling flat plate, an airfoil based on the flying snake \textit{Chrysopelea paradisi}, and experimental measurements of a NACA-0012 airfoil.

\subsubsection{Falling flat plate}
The motion of a falling flat plate has been researched extensively in the context of insect flight, falling leaves, and falling paper \cite{wang2004unsteady,andersen_analysis_2005,andersen_unsteady_2005,ern2012wake,tam2010tumbling}. It provides a simple shape which can exhibit a wide range of behaviors from varying only a few parameters. For a holistic look at this problem, a pair of papers by Andersen, Pasavento, and Wang \cite{andersen_unsteady_2005,andersen_analysis_2005} investigated this problem experimentally and computationally to develop a phenomenological model. In their investigation, the authors use the results of a previous paper \cite{wang2004unsteady} which considered the quasi-steady lift and drag on a flat plate for a range of angles-of-attack and found that lift and drag coefficients can be approximated simply by,
\begin{equation}
	\begin{aligned}
		C_D(\alpha) &= 1.4-\cos\left(2\alpha\right), \\
		C_L(\alpha) &= 1.2\sin\left(2\alpha\right).
	\end{aligned}
\end{equation}
These functions are illustrated in Figure \ref{fig:ExampleComp}b. This model is based on the results of direct numerical simulation for a thin elliptical plate at $Re=100$ \cite{wang2004unsteady}, but represent a simple, analytical expression for lift and drag to develop our methods. Drag is at a minimum where the flat plate is horizontal to the incoming air and at a maximum at $\alpha=90^\circ$, while lift vanishes at $0^\circ$ and $90^\circ$, while reaching a maximum at $45^\circ$. Using this model for lift and drag in the equations of motion found in (\ref{eq:inertial-v1}), we can analyze the velocity space for a single flat plate falling through a fluid due to gravity, shown in Figure \ref{fig:ExampleComp}c for $\theta=-5^\circ$.

\begin{figure}[t]
	\centering
	\includegraphics[width=\textwidth]{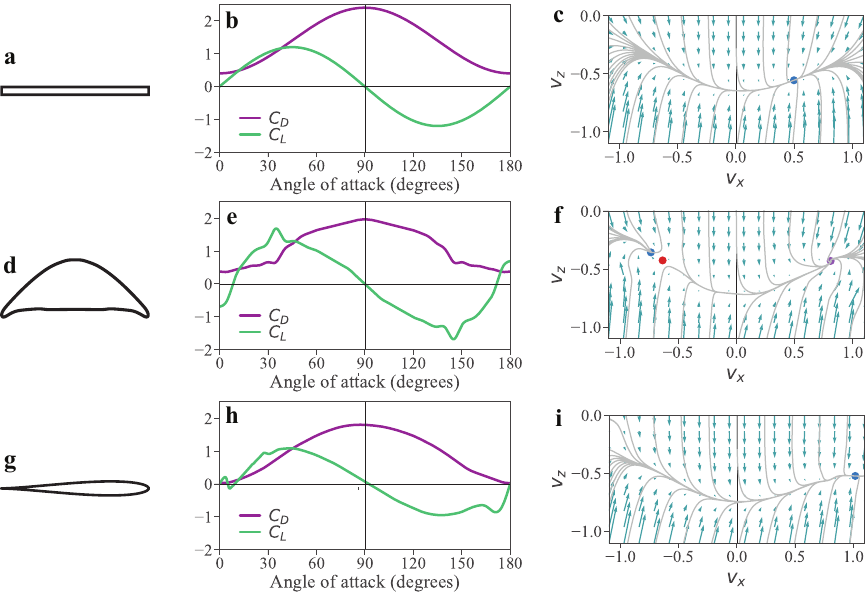}
	\caption{Comparison of the example airfoils considered in this paper. For each airfoil, the lift and drag curves are shown over the interval $\alpha \in (0^\circ, 180^\circ)$. The symmetry of drag and antisymmetry of lift for the flat plate (panel b) and flying snake airfoils (panel e) about $\alpha=90^\circ$ are evident. The NACA airfoil exhibits this same symmetry and anti-symmetry, but about $\alpha=0^\circ$ in panel h. In panels c, f, and i, the acceleration at each velocity is shown by the blue arrows and example trajectories are shown in gray for an example with a fixed pitch angle of $\theta=-5^\circ$.}
	\label{fig:ExampleComp}
\end{figure}

\subsubsection{The flying snake airfoil: the body shape of \textit{Chrysopelea paradisi}}
As a biologically motivated example, we consider the cross-sectional body shape for \textit{Chrysopelea paradisi}, pictured in Figure \ref{fig:ExampleComp}d \cite{socha2011gliding}. During the glide, the snake expands its ribs to form an airfoil-like shape that is horizontally symmetric. Holden et al. \cite{holden_aerodynamics_2014} determined the aerodynamic characteristics of this shape by 3-D printing the extruded cross-section and measuring its lift and drag in a water channel. Measurements were conducted over angles of attack from $-10^\circ$ to $60^\circ$. A variety of methods have been applied to understand and model this animal's behavior \cite{socha2002kinematics,socha_non-equilibrium_2010,jafari_theoretical_2014,krishnan2014lift,socha2015animals,yeaton_global_2017,jafari2017control}. Here, we analyze how a fixed glider with the snake's characteristics would behave, similar to the work of Jafari et al.\ \cite{jafari_theoretical_2014}. This simplified analysis will give insight into the more complex behavior of the snake itself. Its lift and drag characteristics are shown in Figure \ref{fig:ExampleComp}e, and the resulting phase space is shown in Figure \ref{fig:ExampleComp}f for $\theta=-5^\circ$.

\subsubsection{The NACA-0012 airfoil}
As a representative of engineered gliding systems, we consider the example of a NACA-0012 airfoil. It is a vertically symmetric airfoil used in a variety of aircraft with a maximum thickness equal to 12\% of its chord length, as shown in Figure \ref{fig:ExampleComp}g. Its lift and drag characteristics, shown in Figure \ref{fig:ExampleComp}h, were measured in a wind tunnel for angles of attack ranging from $0^\circ$ to $180^\circ$ for application to vertical axis wind turbines \cite{sheldahl1981aerodynamic}. Lift increases near $\alpha=0^\circ$ before it drops off at the point of stall, then, lift continues to increase until approximately $\alpha=45^\circ$. The resulting phase space is shown in Figure \ref{fig:ExampleComp}i for $\theta=-5^\circ$.



\subsection{Shape symmetry and force coefficients}
There are a number of symmetries present in the examples that we have chosen to study here. Both the snake airfoil and flat plate exhibit left-right symmetry, the NACA-0012 airfoil and flat plate exhibit top-bottom symmetry, and the flat plate also shows $180^\circ$ rotational symmetry. Each of these symmetries has natural consequences for the functional symmetries of the fluid force coefficients.

The symmetry of rotation by $180^\circ$, as observed in the flat plate example, means that the shape at angle-of-attack $\alpha$ is the same as the shape at $\alpha+180^\circ$ for all angles of attack. Therefore,
\begin{equation}
	\begin{aligned}
		C_L(\alpha) &= C_L(\alpha+180^\circ), \\
		C_D(\alpha) &= C_D(\alpha+180^\circ).
	\end{aligned}
\end{equation}
These coefficients are therefore cyclic with period $180^\circ$ rather than $360^\circ$ for all other shapes. This fact causes the above model for the flat plate to depend on sinusoidal functions of $2\alpha$.

The top-bottom symmetry of the NACA-0012 airfoil and flat plate means that drag is the same whether the airfoil is pitched up or down, and that lift is exactly opposite for upward or downward pitch. These properties correspond to the properties of even and odd functions, respectively. Therefore,
\begin{equation}
	\begin{aligned}
		C_L(\alpha) &= -C_L(-\alpha), \\
		C_D(\alpha) &= ~~C_D(-\alpha).
	\end{aligned}
\end{equation}
As a corollary, we find that $C_L(0)=0$ for systems with this symmetry. This result follows naturally from $C_L(\alpha)=-C_L(-\alpha)$ when $\alpha=0$.

Finally, the left-right symmetry of the snake and the flat plate is equivalent to the top-bottom symmetry, but rotated by $90^\circ$.
\begin{equation}
	\begin{aligned}
		C_L\left(90^\circ+\alpha\right) &= -C_L\left(90^\circ-\alpha\right), \\
		C_D\left(90^\circ+\alpha\right) &= ~~C_D\left(90^\circ-\alpha\right).
	\end{aligned}
\end{equation}
Following the same logic as above, it's clear that $C_L(90^\circ)=0$ for systems with left-right symmetry.

\section{Equilibrium points of the system}\label{s:equilibria}
We begin our analysis of the phase space structure of the glider model by looking at the possible equilibrium glide points of the system. A criterion for these equilibrium points may be found in the expression for horizontal acceleration $\dot{v}_x$ of the glider from (\ref{eq:inertial-v1}). Each fixed point must correspond with zero acceleration, and therefore, $v^2\left(C_L\sin\gamma-C_D\cos\gamma\right)=0$. This expression is only zero when $v=0$ or $C_L\sin\gamma-C_D\cos\gamma=0$. When $v=0$, the vertical acceleration is given by gravity, $\dot{v}_z=-1$, so this point does not correspond to a fixed point. Therefore, any equilibrium points require the following condition on glide angle $\gamma^*$,
\begin{equation}\label{eq:fixed-point-condition}
	\cot\gamma^* = \frac{C_L}{C_D}\left(\gamma^*+\theta\right).
\end{equation}
The condition for the magnitude of velocity at the fixed point is found by setting $\dot{v}_z=0$ in (\ref{eq:inertial-v1}), which gives the $v_z$-nullcline, which will be discussed in more detail in Section \ref{ss:nullcline}.

There will always exist one fixed point on the interval $\gamma \in (0,\pi)$, the lower half plane of velocity space. To show this, we must make two physical inferences, one regarding lift and one regarding drag, which we have already introduced in Section \ref{ss:liftdrag}. Since the lift coefficient is a mathematical model of a real fluid force, an infinite lift coefficient would be physically unreasonable, and therefore this function maps to the finite interval $\mathcal{I}_L\subset\mathbb{R}$ as in (\ref{eq:liftCoeff}). Secondly, as discussed above, in a quiescent field of fluid with a quasi-steady fluid force, it is impossible for a single body to generate negative drag, or thrust. Furthermore, there must be at least some viscous drag on a body moving through a fluid. Therefore, although it may be small, the range of the drag coefficient, $\mathcal{I}_D$, must be a subset of the positive reals, as in (\ref{eq:dragCoeff}). Thus, it is physically reasonable to assume that the lift coefficient function is finite, and the drag coefficient function is everywhere positive.

\begin{figure}[t]
	\centering
	\includegraphics[height=2.5in]{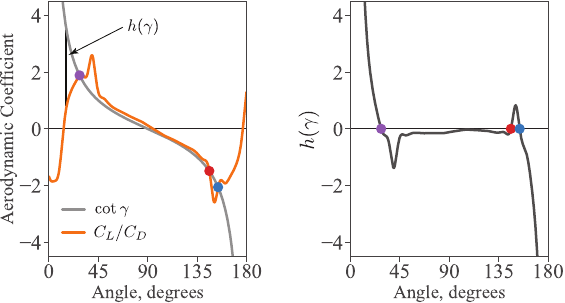}
	\caption{Graphical representation of (\ref{eq:DifferenceEquation}). Every intersection between the two functions, or equivalently, every zero crossing of $h(\gamma)$, represents a fixed point of the system. This example is for the flying snake airfoil at pitch angle $\theta=-5^\circ$, matching the phase space shown in Figure \ref{fig:ExampleComp}f.}
	\label{fig:ProofFigure}
\end{figure}

\begin{theorem}
	For a two-dimensional glider model as conceptualized in (\ref{eq:inertial-v1}) with a smooth function with positive image for drag coefficient and a smooth function with finite range for lift coefficient, there must be at least one fixed point on the open domain $D=\{\gamma~\vert~\gamma \in (0,\pi)\}$.
\end{theorem}

\begin{proof}
	Consider the condition for equilibrium points given above in (\ref{eq:fixed-point-condition}). This condition can be arranged to create a function $h:D \rightarrow \mathbb{R}$.
	\begin{equation}\label{eq:DifferenceEquation}
		h(\gamma) = \cot\gamma-\frac{C_L}{C_D}\left(\gamma+\theta\right)
	\end{equation}
	Any value of $\gamma$ such that $h(\gamma)=0$ corresponds to a fixed point $\gamma^*$ satisfying (\ref{eq:fixed-point-condition}).
	
	From our assumptions about the lift and drag coefficient, that lift is finite and drag is positive and therefore nonzero, we may infer that the lift to drag ratio $\frac{C_L}{C_D}\left(\gamma+\theta\right)$ is itself finite. Thus, our function $h(\gamma)$ is dominated by the contribution of $\cot\gamma$ at both endpoints of the domain.
	\begin{equation}
		\begin{aligned}
			\lim\limits_{\gamma\rightarrow 0^+}h(\gamma) &= \lim\limits_{\gamma\rightarrow 0^+}\cot\gamma \to ~\infty \\
			\lim\limits_{\gamma\rightarrow \pi^-}h(\gamma) &= \lim\limits_{\gamma\rightarrow \pi^-}\cot\gamma \to -\infty
		\end{aligned}
	\end{equation}
	Since $h(\gamma) \rightarrow \infty$ as $\gamma\to 0$ and $h(\gamma)\rightarrow -\infty$ as $\gamma\to\pi$, by the Intermediate Value Theorem (IVT), there must be a point in between such that $h(\gamma)=0$. Therefore, there must be at least one fixed point on the interval $D$.
\end{proof}

By this theorem, there must be some downward ($v_z<0$) equilibrium point for any lift and drag functions obeying the quite general criteria of (\ref{eq:liftCoeff}) and (\ref{eq:dragCoeff}). This function that we have defined, $h(\gamma)$, has the important property of being topologically conjugate to the acceleration along the terminal velocity manifold. Furthermore, we can show that there must be an odd number of fixed points on the interval $D$ outside of edge cases where the function grazes the $0$-line, that is, $h=0$ and $dh/d\gamma=0$.

\begin{theorem}\label{odd_points}
	Outside of those cases where $h(\gamma^*)=0$ and $\left.dh/d\gamma\right|_{\gamma=\gamma^*}=0$, it is guaranteed that there are an odd number of equilibrium points on the domain $D=\{\gamma~\vert~\gamma \in (0,\pi)\}$. 
\end{theorem}

\begin{proof}
Assume that there at least two equilibrium points $\gamma_1^*,\gamma_2^*\in (0, \pi)$ with $\gamma_1^*<\gamma_2^*$ and that $\left.dh/d\gamma\right|_{\gamma=\gamma_1^*}<0$.

If $\left.dh/d\gamma\right|_{\gamma=\gamma_2^*}>0$, then $h(\gamma_2^*+\epsilon)>0$ for a small $\epsilon>0$. Since, from the previous proof, $h(\pi)\rightarrow -\infty$, there must be a third fixed point $\gamma_3^*$ on the interval $(\gamma_2^*, \pi)$ by the IVT. Conversely, if $\left.dh/d\gamma\right|_{\gamma=\gamma_2^*}<0$, then $h(\gamma_2^*-\epsilon)>0$ for a small $\epsilon>0$. Under our assumption, $\left.dh/d\gamma\right|_{\gamma=\gamma_1^*}<0$ which implies $h(\gamma_1^*+\epsilon)<0$. In this case, there must be a third fixed point $\gamma_3^*$ on the interval $(\gamma_1^*, \gamma_2^*)$ by the IVT.

This same logic applies if we assume $\left.dh/d\gamma\right|_{\gamma=\gamma_2^*}<0$ and consider both cases of $\left.dh/d\gamma\right|_{\gamma=\gamma_1^*}$. Therefore, if there are at least two equilibrium points with nonzero derivatives of $h(\gamma)$, there must be a third. By the same argument, four equilibrium points implies a fifth and so on.
\end{proof}

Furthermore, if a bifurcation occurs which transitions from one fixed point to three, the two new equilibrium points must have opposite derivatives. That is, $\left.dh/d\gamma\right|_{\gamma=\gamma^*}>0$ for one fixed point and $\left.dh/d\gamma\right|_{\gamma=\gamma^*}<0$ for the other. If we investigate this expression, we find that,
\begin{equation}\label{eq:dh-dgamma}
\left.\frac{dh}{d\gamma}\right|_{\gamma=\gamma^*} = -\csc^2\gamma^* - \left(\frac{C_L}{C_D}\right)'\left(\gamma^*+\theta\right),
\end{equation}
where $(\cdot)'$ denotes the derivative of a function with respect to its argument.

We can use the Jacobian of the equations of motion (see eq.\ \eqref{linearization_matrix_A} in Appendix \ref{appendix}) to express the conditions for different types of equilibrium points. These conditions are expressed in terms of two variables, $\tau=C_L'/C_D+3$ and $\Delta=1+(C_L/C_D)^2+(C_L/C_D)'$ (see also eq.\ \eqref{tau_and_delta} in Appendix \ref{appendix}). Using trigonometric identities and (\ref{eq:fixed-point-condition}), eq. (\ref{eq:dh-dgamma}) can be rearranged to give,
\begin{equation}
\left.\frac{dh}{d\gamma}\right|_{\gamma=\gamma^*} = -1 - \left(\frac{C_L}{C_D}\right)^2\left(\gamma^*+\theta\right) - \left(\frac{C_L}{C_D}\right)'\left(\gamma^*+\theta\right) = -\Delta.
\end{equation}
If $\Delta<0$, the fixed point must be of saddle type. Therefore, if there exist more than one fixed point, $\gamma^*_i$, on the interval $\gamma\in(0,\pi)$ ordered such that $\gamma^*_1<\gamma^*_2<\cdots$, then every even fixed point must be of saddle type. This can be seen below in Figure \ref{fig:Bifurcation}. Along every vertical slice, there are an odd number of fixed points, and anytime there are more than one, the even fixed points are saddle points, denoted by the red points.

From the criterion given by (\ref{eq:fixed-point-condition}), we show the numerically computed bifurcation diagrams for the three examples considered in this work below in Figure \ref{fig:Bifurcation}. These bifurcation diagrams are numerically found by pseudo-arclength continuation, and show the critical glide angle for each fixed point, $\gamma^*$, as a function of the pitch angle parameter $\theta$. The colors indicate the equilibrium type, with blue points signifying stable nodes, purple points signifying sable foci, and red points signifying saddle points.

Note that this proof only deals with the appearance and number of fixed points, and does not preclude the possibility of Hopf bifurcations, which are possible within this model and discussed in Appendix \ref{appendix}. 
The global bifurcation which connects the terminal velocity manifold and limit cycles is beyond the scope of the present paper and 
is left to future work.
In fact, for the three example airfoils considered here, no Hopf bifurcation occurs, although it could be possible for other airfoils \cite{yeaton_global_2017}.

\begin{figure}[t]
\centering
\includegraphics[width=6.5in]{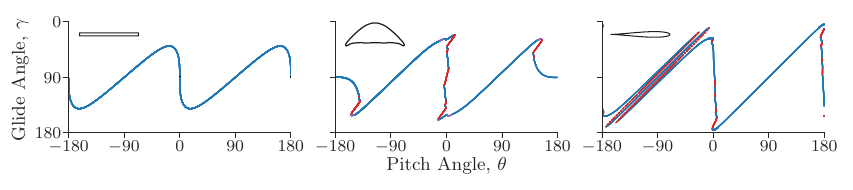}
\caption{The bifurcation diagram showing the glide angle of the equilibrium points $\gamma^*$ at each pitch angle $\theta$ for the flat plate(left), the flying snake cross section(center), and a NACA-0012 airfoil(right). The $\gamma$-axis is flipped such that forward equilibrium glides ($v_x>0$, $\gamma^*<90^\circ$) are located at the top while backward equilibrium glides ($v_x<0$, $\gamma^*>90^\circ$) are located in the lower half of each panel. The type of fixed point is indicated by the color. Blue indicates a stable node, purple indicates a stable focus, red indicates a saddle point, and black indicates a center fixed point.}
\label{fig:Bifurcation}
\end{figure}

\section{Detecting the terminal velocity manifold}\label{s:TVMdetection}
With a better understanding about the behavior of changes in fixed points, we turn our attention to the terminal velocity manifold. From the phase space of the system shown in the right panels of Figure \ref{fig:ExampleComp} (c, f, and i), we can observe several properties of the TVM. As mentioned in the introduction, the TVM is an example of an attracting normally hyperbolic invariant manifold \cite{wiggins2013normally}, meaning that at every point on the manifold, the eigenvalue of the linearized system normal to the manifold has a negative real part. Based on the properties of this structure, we employ a variety of methods to identify it. 

All fixed points of the system lie along the TVM, we will investigate this structure as it relates to the stable and unstable manifolds of fixed points \cite{koon2008dynamical}. Additionally, as shown schematically in Figure \ref{fig:AttractingSubmanifold}, motion along the manifold is slower than motion onto the manifold, giving it the behavior of a slow manifold \cite{kuehn2016multiple}, so we will first approximate it using the $v_z$-nullcline, which serves as a proxy to a critical manifold. As discussed in Section \ref{ss:terminalVelocity}, the TVM acts as a barrier to transport in the velocity space, so we next employ a bisection method to find this structure. Finally, as the TVM is a globally attracting structure, we apply a method based on the attracting behavior of trajectories called the trajectory-normal repulsion rate \cite{haller_variational_2011,nave2019trajectoryfree}.


\subsection{Stable manifold expansion}
As we have shown in Section \ref{s:equilibria}, there must always be at least one equilibrium point on the interval $(0, \pi)$ and always an odd number of equilibrium points. The difference in time scales of motion onto and along the TVM is related to the magnitude of the eigenvalues of the Jacobian at the equilibrium points. The global terminal velocity manifold, then, is associated with the stable and unstable manifolds of the fixed points \cite{krauskopf2005survey}. Figure \ref{fig:AttractingSubmanifold} shows a schematic of the stable and unstable manifolds for an example system with two stable nodes and one saddle point. 

\begin{figure}[t]
	\centering
	\includegraphics[width=0.7\textwidth]{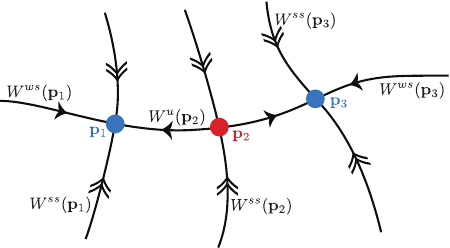}
	\caption{A schematic representation of the terminal velocity manifold as an attracting manifold in the context of the union of the unstable and weakly stable manifolds of saddle ($p_2$) and stable node points ($p_1,p_3$), respectively.}
	\label{fig:AttractingSubmanifold}
\end{figure}

For stable nodes, such as $p_1$ and $p_3$ in the figure, the stable manifold $W^s$ is two-dimensional. However, the stable manifold contains two one-dimensional embedded submanifolds, which are the strong stable $W^{ss}\subset W^{s}$ and weak stable $W^{ws}\subset W^{s}$ submanifolds of each point. These are the nonlinear expansions of the strong stable $\ee_{ss}$ and weak stable $\ee_{ws}$ eigenvectors of the fixed point (see Appendix \ref{stable_node} for details). The associated eigenvalues must have the ordering $\lambda^{ss}<\lambda^{ws}<0$. For the saddle points, such as $\mathbf{p}_2$, there remains an ordering of the magnitudes of the eigenvalues $0<\lambda^u<-\lambda^{ss}$ (see Appendix \ref{saddle} for details). However, in this case, the strong stable submanifold represents the entirety of the stable manifold $W^{ss}=W^{s}$. From Section \ref{s:equilibria}, the even fixed points are saddle points. The terminal velocity manifold is, in general, the union of the weak stable manifolds of all stable equilibrium points in the system.

The most logical method for extracting the TVM, then, is the semi-analytical method of integrating the weak stable eigendirection $\ee_{ws}$ of each stable fixed point backward in time \cite{koon2008dynamical,naik2017geometry}. However, because of the dominance of the strong stable eigenvalue $\lambda^{ss}<\lambda^{ws}$, in backward time the strong stable direction becomes a strong unstable direction. Any deviation from the weak stable submanifold $W^{ws}$ in backward time will lead to the trajectory effectively peeling off of the weak stable submanifold. Integrating the unstable manifold of saddle fixed points $W^u$ in forward time is able to identify the TVM between any stable fixed points. One can also find the TVM via a higher order expansion, which also provides the one-dimensional vector field along the TVM (see Appendix \ref{expansion} for details). However, one may need to go to an unrealistically high order to get the expansions to converge, so to take a global approach to the TVM, a method other than manifold expansion is required.

In the prior study by Yeaton et al.\ \cite{yeaton_global_2017}, the authors present the low order polynomial expansion of the unstable manifold in the neighborhood of equilibrium points in this system and the acceleration along it. This local approach is very successful at capturing the terminal velocity manifold near a fixed point. However, it is not able to accurately predict the TVM further from the fixed point, where higher order terms may no longer be neglected. Therefore, it is necessary to find a global approach for calculating the TVM if we are to analyze how the terminal velocity manifold alters with changes in both our functional lift and drag parameters and pitch parameter.

\subsection{The $v_z$-nullcline}\label{ss:nullcline}
The TVM shows a separation in time scales of motion, as schematically shown in Figure \ref{fig:AttractingSubmanifold}, giving it the structure of an attracting slow manifold. However, the system given by (\ref{eq:inertial-v1}) has no explicit slow parameter as in classical examples with slow-fast dynamics \cite{fenichel1979geometric,jones1995geometric,kuehn2016multiple}. The search for an implicit slow parameter is left to possible future work. As a first step to identifying the terminal velocity manifold with a global approach, we follow Yeaton et al.\ \cite{yeaton_global_2017} and present the $v_z$-nullcline as an initial approximation. Although there is no explicit slow parameter in this system, this resembles the calculation of the critical manifold in a slow-fast system \cite{kuehn2016multiple}. Similar to a critical manifold, the $v_z$-nullcline remains near the attracting manifold observed from trajectories, but the two do not necessarily coincide.

The $v_z$-nullcline may be found by setting $\dot{v}_z=0$. The locus of points can be calculated directly from the second part of (\ref{eq:inertial-v1}),
\begin{equation*}
\dot{v}_z = v^2(C_L\cos\gamma + C_D\sin\gamma) - 1 = 0.
\end{equation*}
Rearranging this equation gives a straightforward expression written in terms of the tangential-normal coordinates for convenience,
\begin{equation}\label{eq:vznullcline}
v = \left(C_L\cos\gamma + C_D\sin\gamma\right)^{-\frac{1}{2}}.
\end{equation}
This can be written parametrically from the definition of $v_x$ and $v_z$ as
\begin{equation}\label{eq:vznullParametric}
\begin{aligned}
v_x &= ~~\left(C_L\cos\gamma + C_D\sin\gamma\right)^{-\frac{1}{2}}\cos\gamma, \\
v_z &= -\left(C_L\cos\gamma + C_D\sin\gamma\right)^{-\frac{1}{2}}\sin\gamma.
\end{aligned}
\end{equation}
We consider the $v_z$-nullcline to be given as the range of glide angles $\gamma$ which satisfy (\ref{eq:vznullcline}) between singular values. The singular values occur where the denominator of (\ref{eq:vznullcline}) goes to zero, given by
\begin{equation}\label{eq:singularvznullcline}
\gamma_{s}=\arctan\left(-\frac{C_L}{C_D}\right).
\end{equation}

\begin{figure}[t]
\centering
\includegraphics[width=6.5in]{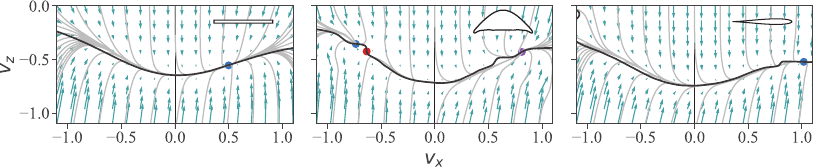}
\caption{The $v_z$-nullcline as an approximation to the terminal velocity manifold for the flat plate=(left), the flying snake cross section(center), and a NACA-0012 airfoil(right) at pitch angle of $\theta=-5^\circ$. The nullcline remains close to the most attracting curve, but does not lie along it.}
\label{fig:nullcline}
\end{figure}

The $v_z$-nullclines compared against the trajectories for all three of our example gliders are shown in Figure \ref{fig:nullcline}. As previously discussed, the TVM is the curve onto which all trajectories collapse. However, in Figure \ref{fig:nullcline}, it is clear that many trajectories pass through the $v_z$-nullcline. Therefore, although the TVM and $v_z$-nullcline are close to one another, they are not the same curve. As the terminal velocity itself is an invariant manifold on which motion is slow and all acceleration is tangential and not generally perpendicular to the $v_z$ direction, then the TVM is not generally the $v_z$-nullcline. That is, along the TVM vertical acceleration is nonzero in general, even though  it is small $\left|\dot{v}_z\right| \ll 1$. The $v_z$-nullcline approximates the terminal velocity manifold, but is inexact.

\subsection{Bisection method}
Based on the observation that the TVM acts as a barrier to transport and its repelling nature in backward time, we introduce a bisection method to numerically identify the TVM. Bisection methods are typically used for identifying a zero-crossing of a function over a fixed interval. They are conceptually straightforward algorithms that have been extended for application to a variety of problems in dynamical systems to determine the boundaries between basins of stability \cite{battelino1988multiple,skufca2006edge}.

For a one dimensional function, one can find the zero-crossing of a function by beginning with endpoints on either side of the zero-crossing and evaluating the sign of the function at the midpoint of the two endpoints. If the function at the midpoint is negative, then the midpoint replaces the lower endpoint. If positive, the midpoint replaces the upper endpoint. The midpoint of the new endpoints is calculated and the process is repeated until the endpoints are within some tolerance of one another.

\begin{figure}[t]
	\centering
	\includegraphics[width=\textwidth]{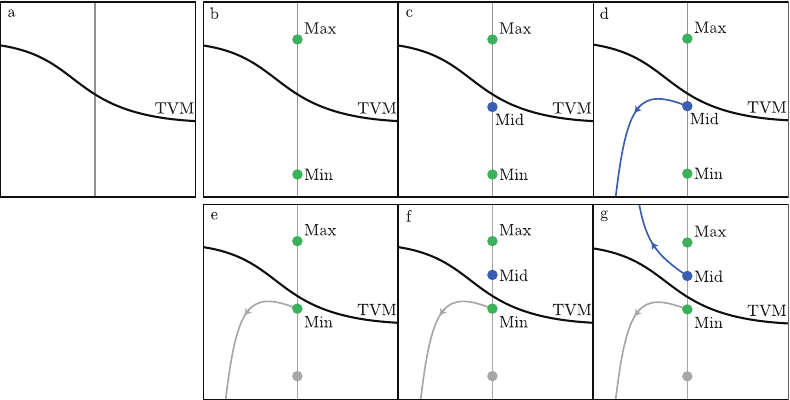}
	\caption{Schematic representation of a bisection algorithm to find the true TVM, shown in black, along a vertical slice of velocity space shown in gray (a). We select two points that bracket the manifold, labeled Max and Min and shown in green (b). Next, we calculate the midpoint of Max and Min, labeled Mid and shown in blue (c). We integrate Mid backward in time (d) and find that the trajectory moves downward. Therefore, Mid is selected as the new Min (e), a new Mid is selected (f) and integrated backward (g). This process is then repeated iteratively until the distance between Max and Min is smaller than a specified tolerance.}
	\label{fig:BisectAlgorithm}
\end{figure}

The implementation of a bisection method in our context is based on the origin of trajectories rather than the value of a function. Selecting a point in phase space and integrating backward in time, we check to see whether the trajectory heads toward positive values of vertical velocity or negative values of vertical velocity, which indicates whether the trajectory which crosses our test point began above or below the TVM. We select test points along a vertical slice of velocity space with a fixed initial horizontal velocity and do a bisection search for the corresponding $v_z$-value. The schematic of our bisection method based on this classification is shown in Figure \ref{fig:BisectAlgorithm}. An initial Max and Min are selected above and below the true TVM (Fig. \ref{fig:BisectAlgorithm}b). Then the midpoint, Mid, is selected (Fig. \ref{fig:BisectAlgorithm}c) and, by backwward integration, found to be below the TVM (Fig. \ref{fig:BisectAlgorithm}d). This point replaces Min (Fig. \ref{fig:BisectAlgorithm}e) and a new Mid is selected (Fig. \ref{fig:BisectAlgorithm}f).

\begin{figure}[t]
\centering
\includegraphics[width=6.5in]{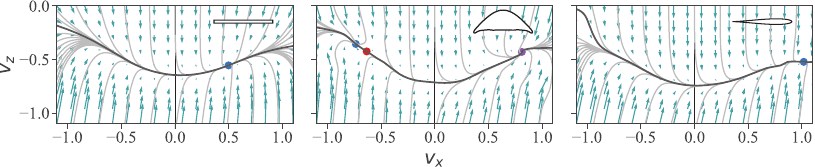}
\caption{The velocity space at a pitch angle of $\theta=-5^\circ$ for the flat plate (left), the flying snake cross section (center), and a NACA-0012 airfoil(right). The bisection method described in Figure \ref{fig:BisectAlgorithm} is able to accurately find the TVM (black), which attracts all trajectories (gray) of the vector field (blue).}
\label{fig:bisection}
\end{figure}

Figure \ref{fig:bisection} shows the results of our implementation of the bisection algorithm. A point along the terminal velocity manifold was found outside of the boundary of the figure through bisection and integrated forward in time. For the flying snake example in the center panel, the saddle fixed point was also calculated, and its unstable manifold integrated in forward time. Using the bisection method from the outside in conjunction with unstable manifold expansion from within the manifold provides a piecewise global approach to find the manifold. It is also possible, although more computationally expensive, to conduct a bisection search at a variety of points $v_x$ across the entire domain to find the corresponding point $v_z$ that lies along the manifold. These two approaches to the bisection algorithm give identical results.

We find that this method is very successful in identifying the TVM. Visual inspection of Figure \ref{fig:bisection} shows that all trajectories go to the calculated TVM in black. The success of this method confirms the observation that the TVM serves as the boundary between trajectories with large negative initial vertical velocities and trajectories with zero or positive initial vertical velocities.

\subsection{Trajectory-normal repulsion rate}
Another method, the \textit{trajectory-normal repulsion rate}, provides additional physical insight for the whole system  \cite{haller_variational_2011,nave2019trajectoryfree}. This quantity gives a measure of how much an invariant manifold normally repels nearby trajectories over a finite time, $T$. In an autonomous system, every trajectory is an invariant manifold, so the repulsion rate gives a scalar value at every point $x_0$ that indicates how much nearby trajectories are normally repelled. As illustrated in Figure \ref{fig:RepFactor}, this value is given by the forward mapping of trajectory-normal vectors,
\begin{equation}\label{eq:repfactor}
\rho_T = \langle \mathbf{n}_T, \nabla F_T\mathbf{n}_0\rangle,
\end{equation}
\begin{figure}[b!]
\centering
\includegraphics[width=0.6\textwidth]{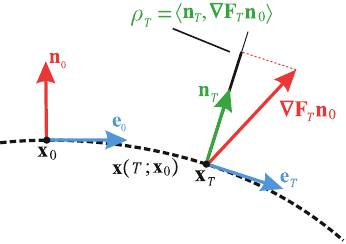}
\caption{Graphical explanation of the trajectory-normal repulsion rate, $\rho_T$. The initial unit normal vector $\mathbf{n}_0$ is mapped forward by the gradient of the flow map $\nabla \mathbf{F}_T\mathbf{n}_0$. By taking an inner product with the new unit normal vector $\mathbf{n}_T$, we measure the stretching of phase space normal to the trajectory of $\mathbf{x}_0$ over the time $T$. Reproduced from Nave et al. \cite{nave2019trajectoryfree}.}
\label{fig:RepFactor}
\end{figure}
where $F_T$ and $\nabla F_T$ represent the flow map of the system over the interval $(0,T)$ and its gradient, $\mathbf{n}_0$ is unit the normal vector at time $0$, and $\mathbf{n}_T$ is the unit normal vector at time $T$. The normal vectors are calculated by $90^\circ$ counterclockwise rotation,
\begin{equation}
\mathbf{R}=\left(\begin{array}{cc}
  0 & -1 \\ 1 & ~~0
        \end{array}\right),
\end{equation} 
of the tangent vector given by the acceleration $\dot{\mathbf{v}}$ in our problem, normalized by the magnitude of that acceleration $\left|\dot{\mathbf{v}}\right|$: $\mathbf{n}=\mathbf{R}\dot{\mathbf{v}}/\left|\dot{\mathbf{v}}\right|$.

In a Lagrangian method such as this one, globally attracting features in the system are found by detecting repelling features in backward time. In Figure \ref{fig:RepComp}, we show a comparison of the trajectory-normal repulsion rate over an integration time of $T=-0.32$, expressed in the rescaled time of (\ref{eq:dimensionlessEOM}). This integration time was the longest computational time for which no integration for an initial condition in the domain failed using the LSODA integration pack through SciPy \cite{jones2014scipy}. In backward time, the squared dependence on velocity causes each integration step to get increasingly large.

\begin{figure}[t!]
\centering
\includegraphics[width=\textwidth]{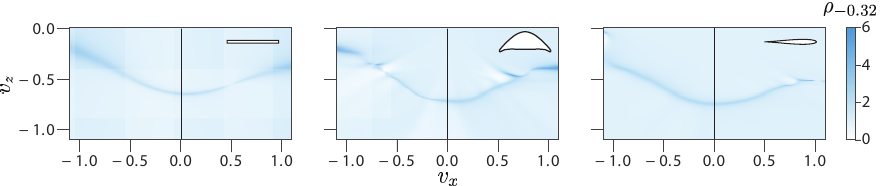}
\caption{The velocity space at a pitch angle of $\theta=-5^\circ$ for the flat plate(left), the flying snake cross section(center), and a NACA-0012 airfoil(right). The colormap shows the value of the trajectory-normal repulsion, described by (\ref{eq:repfactor}) and illustrated in Figure \ref{fig:RepFactor}, of the system when integrated backward. The black curve represents the results of the bisection method described in Figure \ref{fig:BisectAlgorithm}.}
\label{fig:RepComp}
\end{figure}

There are regions, particularly in the velocity space for the flying snake airfoil and the NACA-0012 airfoil, where portions of other trajectories may be more attracting than the nearby TVM over the integration time used. Therefore, we constraint the TVM to be the trajectory which maximizes the backward time trajectory-normal repulsion rate.

\section{Pitch angle dependence of the terminal velocity manifold}\label{s:3DTVM}
The bifurcation diagrams of Figure \ref{fig:Bifurcation} do not capture how the terminal velocity manifold itself changes with respect to the pitch angle. Therefore, we look at how the terminal velocity manifold changes with pitch angle and visually represent this change in extended phase space. We then prescribe pitch angle control and observe the classical behaviors of gliding and fluttering as occurring along the extended terminal velocity manifold.

\subsection{The terminal velocity manifold in extended phase space}

\begin{figure}[t]
	\centering
	\includegraphics{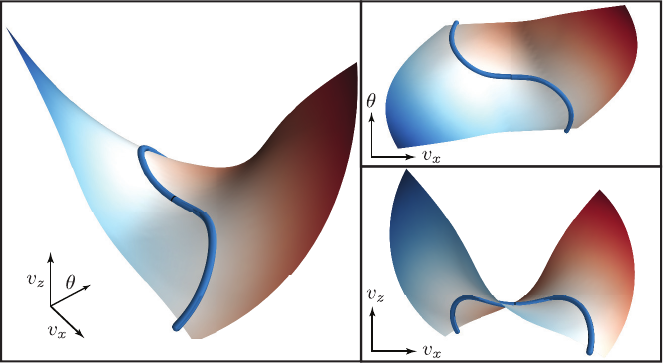}
	\caption{The terminal velocity manifold as a two-dimensional surface embedded in 3-dimensional space for the flat plate for parameter values $\theta \in [-45^\circ, 45^\circ]$. Blue values indicate positive acceleration along the manifold and red values indicate negative acceleration along the manifold, while the one-dimensional curve shows the equilibrium points of the system, including stable nodes (blue) and the center equilibrium point (black, at the center of the manifold).} 
	\label{fig:3Dtvm-plate}
\end{figure}

To consider the effects of the pitch parameter $\theta$ on the TVM, we look to a 3-dimensional extended phase space, including pitch angle as an independent variable without motion. This view of the system allows us to visualize changes of the manifold in the parameter direction while still maintaining the same dynamics of the equations of motion (\ref{eq:inertial-v1}). We re-cast the system into extended phase space with the following equations,
\begin{equation}
\begin{aligned}
\dot{v}_x &= v^2(C_L\left(\gamma+\theta\right)\sin\gamma - C_D\left(\gamma+\theta\right)\cos\gamma), \\
\dot{v}_z &= v^2(C_L\left(\gamma+\theta\right)\cos\gamma + C_D\left(\gamma+\theta\right)\sin\gamma) - 1, \\
\dot{\theta} &= 0.
\end{aligned}
\label{eq:extended phase space}
\end{equation}
In this model, every fixed point of the two-dimensional system remains a fixed point because of the negligible dynamics in the $\theta$-direction. Therefore, the one-dimensional TVM will become two-dimensional as it is extended in the $\theta$-direction. As the accelerations $\dot{v}_x$ and $\dot{v}_z$ depend smoothly on $\theta$, we will be able to uncover a smooth extended TVM. To visualize this, we calculate the one-dimensional TVM from the two-dimensional model of (\ref{eq:inertial-v1}) using the bisection method over a variety of pitch angles and stitch these together to form a two-dimensional surface in the extended, 3-dimensional model. The resulting surfaces are shown in Figures \ref{fig:3Dtvm-plate}, \ref{fig:3Dtvm-snake}, and \ref{fig:3Dtvm-naca} over the interval $\theta \in [-45^\circ, 45^\circ]$.

\begin{figure}[t]
	\centering
	\includegraphics{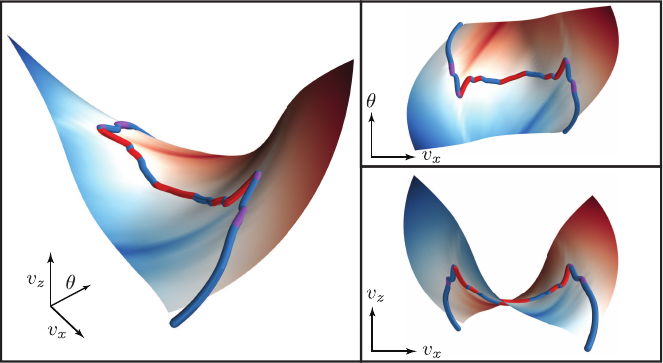}
	\caption{The terminal velocity manifold as a two-dimensional surface embedded in 3-dimensional space for the flying snake airfoil over a pitch domain of $\theta \in [-45^\circ, 45^\circ]$. Blue values indicate positive acceleration along the manifold and red values indicate negative acceleration along the manifold, while the one-dimensional curve shows the equilibrium points of the system, including stable nodes (blue), saddle points (red), and stable foci (purple).}
	\label{fig:3Dtvm-snake}
\end{figure}

\begin{figure}[t]
	\centering
	\includegraphics{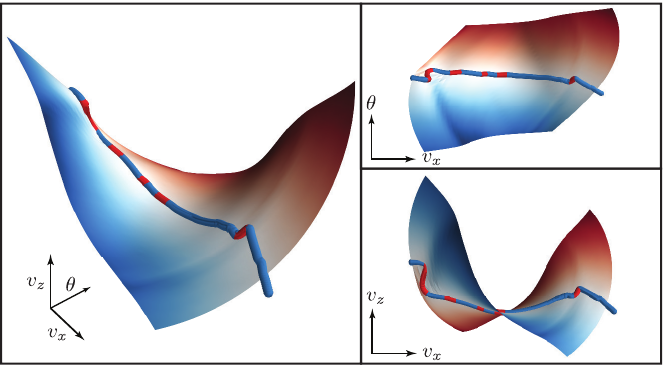}
	\caption{The terminal velocity manifold as a two-dimensional surface embedded in 3-dimensional space for the NACA-0012 airfoil over a pitch domain of $\theta \in [-45^\circ, 45^\circ]$. Blue values indicate positive acceleration along the manifold and red values indicate negative acceleration along the manifold, while the one-dimensional curve shows the equilibrium points of the system, including stable nodes (blue) and saddle points (red). Equilibrium velocities with a larger horizontal velocity of $\left|v_x\right|>1.5$ were also omitted.}
	\label{fig:3Dtvm-naca}
\end{figure}

In each figure, the colormap on the surface shows the acceleration at every point on the manifold. Blue regions are associated with positive acceleration and red regions are associated with negative acceleration, where the positive direction is associated with the positive $v_x$ axis. As the TVM is an invariant manifold, the vector field is purely tangential to the manifold itself. The equilibrium points of the system are shown along the manifold as well, representing the bifurcation diagram of the system. The colors of Figure \ref{fig:Bifurcation} still hold: blue points are stable nodes, red points are saddle points, and purple are stable foci. For the flat plate considered in Figure \ref{fig:3Dtvm-plate}, there is a single stable node in all cases, except $\theta=0$ which contains a stable center manifold. The saddle-node bifurcations of the flying snake are evident in Figure \ref{fig:3Dtvm-snake}. The left-right symmetry of these two airfoils is also visible in the anti-symmetry of the TVM in the $v_x$ direction about $\theta=0$. The NACA-0012 airfoil has a broad range of equilibria over a very small region, showing its sensitivity to pitch angle. In Figure \ref{fig:3Dtvm-naca}, we only show equilibria with a horizontal velocity magnitude smaller than $1.5$, for comparison at the same scale as the other two TVM figures. As seen in \ref{fig:Bifurcation}, there is a wide range of fixed points for this system on the interval $\theta \in [-45^\circ, 45^\circ]$. Outside of this narrow range, all of those fixed points have a much larger magnitude of velocity.

\subsection{Conceptualizing motion with the terminal velocity manifold}\label{sec:time_varying_pitch}
With the two-dimensional TVM in extended phase space, we now look toward allowing variation in pitch with time. To consider the effects of the fluid moment on the body, it would become necessary to account for pitch rate in the dynamics of the system. Therefore, we will instead specify controlled pitch kinematics, i.e., $\theta(t) = u(t)$, where $u(t)$ is a prescribed control input function, and allow the system to evolve translationally (i.e., in $(v_x,v_y)$ space) through the two-dimensional equations of motion of (\ref{eq:inertial-v1}).  Thus, instead of \eqref{eq:extended phase space}, the {\it controlled extended phase space} equations of motion are,
\begin{equation}
\begin{aligned}
\dot{v}_x &= v^2(C_L\left(\gamma+\theta\right)\sin\gamma - C_D\left(\gamma+\theta\right)\cos\gamma), \\
\dot{v}_z &= v^2(C_L\left(\gamma+\theta\right)\cos\gamma + C_D\left(\gamma+\theta\right)\sin\gamma) - 1, \\
\theta(t) &= u(t).
\end{aligned}
\label{eq:extended-dynamics}
\end{equation}
Note that we are further assuming that $\dot{\theta}=\dot{u}$ is small enough to neglect additional forces which arise from pitch dynamics which are considered by, for instance, Andersen et al. \cite{andersen_analysis_2005,andersen_unsteady_2005}. If the pitch dynamics are slower than the motion onto the manifold, then all motion after an initial transient will occur close to this higher dimensional TVM. With this in mind, we consider the phenomena of gliding flight and fluttering. 

Gliding flight has served as an initial motivation for this model \cite{yeaton_global_2017}. Therefore, to look at how the motion of a gliding body occurs along the TVM, we consider pitch dynamics which slowly increase throughout the motion. This serves to represent the way that animal gliders begin with an initial downward descent and pitch up before landing \cite{socha2015animals,bahlman2013glide,bishop2006relationship}. Next, we look to fluttering descent as considered in a variety of studies on falling seeds, disks, and plates \cite{minami_various_2003,field_chaotic_1997,ern2012wake,andersen_unsteady_2005,andersen_analysis_2005,tam2010tumbling,vincent2016holes}. These pitch dynamics are given by slowly oscillating the pitch angle throughout the motion. In real examples, the oscillating pitch angle is the result of varying fluid moments on the body, so we choose a simple sinusoidal oscillation to represent the resulting kinematics. By considering these example motions, we can visualize the ways in which the translational forces considered in this model contribute to the full physics of a passively descending body.

\subsubsection{Gliding flight}
For animals exhibiting gliding flight, a typical glide includes: (1) an initial, ballistic acceleration; (2) a shallowing glide through the middle of the motion; and finally (3) a landing maneuver in which they slow descent \cite{socha2015animals,bahlman2013glide}. We represent this behavior here with a simulation of a gliding snake airfoil which increases its pitch angle through the glide, starting from an initially downward pitch angle \cite{socha2002kinematics,bishop2006relationship}. This gives a larger initial pitch angle to maximize ballistic acceleration, a shallowing pitch angle through the glide as the animal passes through its maximum lift-to-drag ratio, and finally a pitch up to decelerate overall for landing. A function $u(t)$ which yeilds a linear increase in pitch is the simplest way to represent this phenomenon. The results of this linearly increasing pitch angle can be seen in Figure \ref{fig:gliding}. The hallmarks of the behavior can be found in the velocity space, which shows the projection of motion in the $v_x$-$v_z$ plane. Initially, with a negative pitch angle, motion onto the TVM is rapid, and acceleration occurs quickly. Next, the horizontal velocity increases as the glide moves forward. Finally, as the glider's pitch angle levels out, both vertical and horizontal velocity decrease for a safer landing.

\begin{figure}[t]
\centering
\includegraphics[trim={0in 0in 0.6in 0.3in},clip]{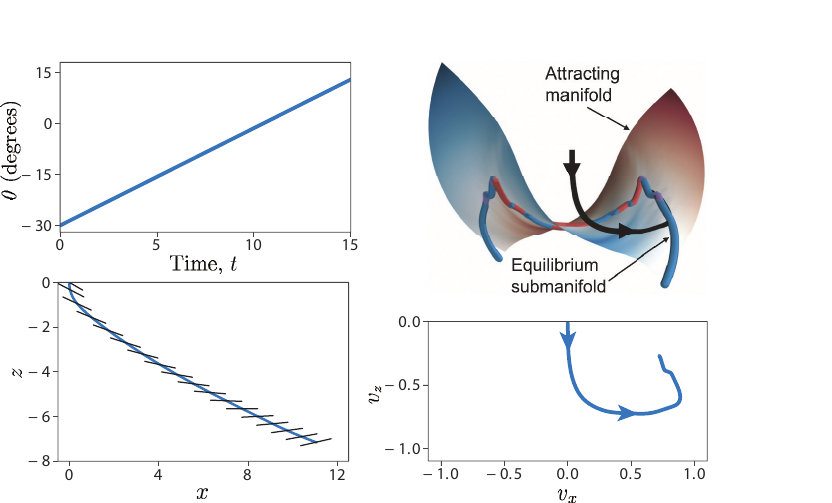}
\caption{In this figure, we show an example glide of a snake-shaped glider with a linearly increasing pitch angle $\theta$ with respect to time, as shown in the top-left corner. The resulting motion in physical space is shown in the bottom-left, with the pitch angle marked by the intersecting lines. This motion is shown in extended phase space with respect to the extended TVM, the attracting manifold, in the top-right, with the black line representing the trajectory. The bottom-right panel shows motion in velocity space. The glider is quickly drawn toward fast downward motion, but as it pitches up, the trajectory moves along the terminal velocity manifold toward a stable forward glide before slowing descent just before landing.}
\label{fig:gliding}
\end{figure}

\subsubsection{Fluttering plates}
The fluttering of a thin body through a fluid has been studied extensively in a variety of studies \cite{field_chaotic_1997,ern2012wake,andersen_unsteady_2005,andersen_analysis_2005,tam2010tumbling,vincent2016holes}. This is a frequently-observed behavior of passively descending plates characterized by coupled oscillations of pitch angle and horizontal motion as the plate descends vertically. We can emulate this behavior with our extended three-dimensional model by prescribing oscillating pitch kinematics which are faster than the motion along the TVM but slower than the time scale of motion onto the manifold.

\begin{figure}[t]
\centering
\includegraphics[trim={0 0 0.4in 0.2in},clip]{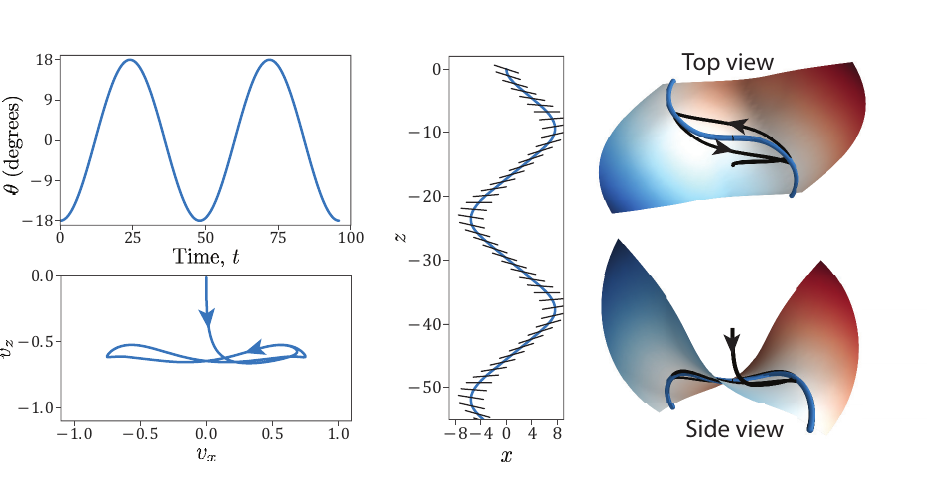}
\caption{In this figure, we show the example of fluttering descent of a flat plate. The controlled pitch angle, $\theta$, varies sinusoidally, as shown in the top-left. The resulting motion in physical space is shown in the central panel, with snapshots of the flat plate shown in black. The motion in extended phase space is shown in the two panels on the right side relative to the terminal velocity manifold. The motion projected into velocity space is shown in the bottom-left corner.}
\label{fig:fluttering}
\end{figure}

With oscillating pitch control, we are able to replicate dynamics which closely resemble classical fluttering. The trajectory forms a limit cycle oscillation which lies  close to the extended TVM, but not exactly along it, as shown in Figure \ref{fig:fluttering}.  A spectral submanifold approach \cite{haller2016nonlinear} might be used to reveal the actual periodically varying two-dimensional surface on which the limit cycle lives, but such a calculation goes beyond the scope of this paper, and we consider the extended TVM an adequate approximation.  We note that the magnitude of the controlled oscillations determine the size of the limit cycle (i.e., larger amplitude for $\theta(t)=u(t)$ leads to larger amplitudes in horizontal excursions).  

Note that this limit cycle in the TVM is driven by periodic motion of the pitch acting as a periodic forcing. It is a different mechanism compared with the limit cycle due to the Hopf bifurcation, a nonlinear phenomenon inherent in the $(v_x,v_y)$ dynamics themselves, discussed in Appendix \ref{hopf}, and dependent on the lift and drag curves. We note that none of the example airfoils used in this study undergo Hopf bifurctaions in pitch, but it may be possible for other airfoils \cite{yeaton_global_2017}, whose lift and drag curves meet the criteria described  in Appendix \ref{hopf}.

\subsubsection{Implications for controlled gliding}

As previous sections indicate, the TVM is a globally attracting invariant manifold that acts as the dominant organizing feature of the extended phase space.  
Within the TVM, there is also an equilibrium submanifold with equilibria of alternating stability type, with different stability basins. By using pitch $\theta$ as a control parameter, one could trigger a transition from one stable equilibrium to another, to achieve controlled trajectories  which achieve certain objectives, such as energy efficiency, maximum travel distance, etc.  Biological systems may already use this pitch dependence, for example, to slow a glide before landing and controlling contact with a substrate (e.g., Figure \ref{fig:gliding}).
As more detailed kinematics data become available from animal studies, this framework can be used to understand how experimentally recorded gliders alter their trajectory through control of body orientation.

The TVM framework also suggests that the glide dynamics of engineered aerial and aquatic autonomous gliders could be designed to exploit the structure of the TVM.
Airfoils or hydrofoils could be designed via an inverse approach of starting with a desired TVM and then designing the foil such that the angle-of-attack dependence of the lift and drag coefficients leads to the desired TVM.  For instance, the desired TVM could be chosen such that an autonomous controlled glider would  only need small actuations in pitch angle to passively switch to different glide states to achieve desired trajectories, or other desired functionality for the glide (e.g., ease of landing, maneuverability).  The framework of \cite{balasuriya2017unsteadily} of could be adapted to design lift and drag functions that achieve a desired TVM.

\section{Summary and Conclusion}\label{s:conclusion}
In the present work, we have taken the observation of a terminal velocity manifold (TVM) from Yeaton et al.\ \cite{yeaton_global_2017} and placed it on a more mathematically rigorous footing. Through  various methods of computing this curve, we have gained insights into its properties. First, it is the union of all weak stable submanifolds of stable equilibrium points. Because of this structure, computational techniques are required. We have employed a bisection method to identify the TVM via dichotomy by integrating trajectories in backward time to find their origin. From this method, we have seen that the terminal velocity manifold divides velocity space into trajectories with initial vertical accelerations aligned with the direction of gravity from those which initially accelerate opposite the direction of gravity. We also calculated the trajectory-normal repulsion rate \cite{haller_variational_2011,nave2019trajectoryfree} in backward time to show that the TVM is the most normally attracting curve in velocity space in forward time. Finally, we show the terminal velocity manifold in parameter-extended phase space to provide visual intuition for the mechanics of gliding flight and passive descent.

The glider model considered in this paper presents a naturally nonlinear model with interesting mathematical properties. The TVM represents an epsilon-free slow-fast system in which there is a separation of time scales without an explicit slow parameter, and the methods discussed in this paper to discover the TVM may have implications for the analysis of other slow-fast systems. This model also presents the challenge of analyzing functional parameters in a system. Through our proof of fixed point conditions, we show one way in which these kinds of functional parameters may be analyzed, deducing information about the system based on the constraints on the space of possible functions. In this case, the physical assumption of finite lift-to-drag ratio gave insight into the possible equilibria of the system. This work provides a new physical intuition into the behavior of gliding bodies, and demonstrates a variety of methods for the computation of influential geometric structures in mathematical models.

\section*{Acknowledgments}
This work was supported by National Science Foundations grants 1537349 and 1821145 and by the Biological Transport (BioTrans) Interdisciplinary Graduate Education Program at Virginia Tech. We would also like to thank Isaac Yeaton and Jake Socha for many fruitful conversations about this work.

\section*{Appendix}

\begin{appendices}

\section{Glider Model Stability Analysis, Limit Cycles, and Terminal Velocity Manifold Analytical Approximations}\label{appendix}

Below we list the 
details about the Hopf bifurcation possible in our model.
The equilibrium condition, from \eqref{eq:dimensionlessEOM} and \eqref{eq:glider model}, implies,
\begin{align*}
    \bar{v}^* =& \dfrac{1}{
         \left( C_L(\alpha^*)^2 + C_D(\alpha^*)^2 \right) ^{1/4} } , \\
    \gamma^* =& \cot^{-1} \left( \dfrac{C_L(\alpha^*)}{C_D(\alpha^*)} \right), \\
    \bar{v}^*_x =& ~~~\bar{v}^* \cos{\gamma^*}, \\
    \bar{v}^*_z =& -\bar{v}^* \sin{\gamma^*}, \\
    \alpha^* =& \theta + \gamma^*. 
\end{align*}


\subsection{Expansion about the equilibrium}\label{expansion}

To obtain an analytical approximation of the eigenvalues and eigenvectors, and to put the system in a form where we can analytically obtain the glide manifold in the snake phase space, we first do a change of coordinates centered on an equilibrium point. 
We will work in polar coordinates, since the equations of motion look simpler,
\begin{equation}\label{transformation_v_gamma_to_psi_r}
    \psi = \gamma - \gamma^*, \qquad
    r = \bar v - \bar v^*
\end{equation}
where we are working in non-dimensional and rescaled variables. At equilibrium, from \eqref{eq:dimensionlessEOM}, we have,
\begin{equation}\label{Omega_zero_eqs_polar_equil}
\begin{split}
\hat v^{\prime} = 0 & \Rightarrow   \bar v^{*2} C_D(\alpha) = \sin{\gamma^*} \\
\gamma^{\prime} = 0 & \Rightarrow   \bar v^{*2} C_L(\alpha)  =  \cos{\gamma^*}.
\end{split}
\end{equation}
where $(\cdot)'$ will be used throughout this appendix to denote the derivative of a function with respect to its argument.
In the shifted coordinates, the equilibrium is the origin and the equations of motion are
\begin{equation}
\begin{split}
\psi^{\prime} & = 
-  (\bar v^* + r)  C_L(\gamma^* + \theta^* + \varphi + \psi) + \dfrac{1}{(\bar v^* + r)} \cos(\gamma^* + \psi),
\\
r^{\prime} & = 
- (\bar v^* + r)^2 C_D(\gamma^* + \theta^* + \varphi+ \psi) + \sin(\gamma^* + \psi),
\end{split}
\end{equation}
We want to write the right-hand-side of the equations of motion as a power series expansion in $\psi$ and $r$. To start out, we will get this expansion to second-order.

Let's first look at the $\psi^{\prime}$ expression.
Note that, via Taylor expansion,
\begin{equation}
\dfrac{1}{(\bar v^* + r)} = \dfrac{1}{\bar v^*\left( 1 + \frac{r}{\bar v^*} \right) } = \dfrac{1}{\bar v^*} \left( 1 - \frac{r}{\bar v^*} + \left( \frac{r}{\bar v^*} \right)^2 
-\left( \frac{r}{\bar v^*} \right)^3
+
\mathcal{O}\left( \frac{r}{\bar v^*} \right)^4\right).
\end{equation}
Using the $\cos$ addition formula,
\begin{equation}
\cos(\gamma^* + \psi) = \cos \psi \cos \gamma^* - \sin \psi \sin \gamma^*,
\end{equation}
along with (\ref{Omega_zero_eqs_polar_equil}), we get,
\begin{equation}
\begin{split}
\frac{1}{\bar v^*} \cos(\gamma^* + \psi) & = 
\frac{1}{\bar v^*}
\left[ 
 \bar v^{*2} C_L(\alpha^*) \cos \psi - 
 \bar v^{*2} C_D(\alpha^*) \sin \psi 
\right], \\
& = 
 \bar v^{*} 
\left[ C_L(\alpha^*) \cos \psi - C_D(\alpha^*) \sin \psi 
\right],
\end{split}
\end{equation}
so,
\begin{equation}
\begin{split}
\dfrac{1}{(\bar v^* + r)} \cos(\gamma^* + \psi)
&= \bar v^{*} 
\left[ C_L(\alpha^*) \cos \psi - C_D(\alpha^*) \sin \psi \right] 
\left( 1 - \frac{r}{\bar v^*} + \left( \frac{r}{\bar v^*} \right)^2 
-\left( \frac{r}{\bar v^*} \right)^3
+\mathcal{O}\left( \frac{r}{\bar v^*} \right)^4\right),
\\
& =
\bar v^{*} 
\left[ C_L(\alpha^*) \cos \psi - C_D(\alpha^*) \sin \psi \right]
-
r 
\left[ C_L(\alpha^*) \cos \psi - C_D(\alpha^*) \sin \psi \right]
\\
&+
\left( \frac{r^2}{ \bar v^{*}} - \frac{r^3}{ \bar v^{*2}}  \right) 
\left[ C_L(\alpha^*) \cos \psi - C_D(\alpha^*) \sin \psi \right]
+
\mathcal{O}\left( r^4 \right).
\end{split}
\end{equation}
Also note that $C_L(\gamma^* + \theta^* + \varphi + \psi)
=C_L(\alpha^* + \psi)$, and by Taylor series expansion we have,
\begin{equation}
C_L(\alpha^* + \psi) = C_L(\alpha^*) + \psi C_L^{\prime}(\alpha^*)
+\tfrac{1}{2}\psi^2 C_L^{\prime \prime}(\alpha^*) + 
\mathcal{O} ( \psi^3 ),
\end{equation}
and similarly for the drag term,
\begin{equation}
C_D(\alpha^* + \psi) = C_D(\alpha^*) + \psi C_D^{\prime}(\alpha^*)
+\tfrac{1}{2}\psi^2 C_D^{\prime \prime}(\alpha^*) 
+\tfrac{1}{6}\psi^3 C_D^{\prime \prime \prime}(\alpha^*)
+ 
\mathcal{O} ( \psi^4 ),
\end{equation}
so,
\begin{equation}
\begin{split}
-  (\bar v^* + r)  C_L(\gamma^* + \theta^* + \varphi + \psi)
&=
-  \bar v^*
\left[C_L(\alpha^*) + \psi C_L^{\prime}(\alpha^*)
+\tfrac{1}{2}\psi^2 C_L^{\prime \prime}(\alpha^*) 
+\tfrac{1}{6}\psi^3 C_D^{\prime \prime \prime}(\alpha^*)
+\mathcal{O} ( \psi^4 ) \right],
\\
&
-  r
\left[C_L(\alpha^*) + \psi C_L^{\prime}(\alpha^*)
+\tfrac{1}{2}\psi^2 C_L^{\prime \prime}(\alpha^*) + 
\mathcal{O} ( \psi^3 ) \right].
\end{split}
\end{equation}
So the $\psi^{\prime}$ expression becomes,
\begin{equation}
\begin{split}
\psi^{\prime} & = 
\bar v^* 
\left( -C_L(\alpha^*+\psi) + C_L(\alpha^*)\cos \psi
- C_D(\alpha^*) \sin \psi \right), \\
&
+  r \left( -C_L(\alpha^*+\psi) - C_L(\alpha^*)\cos \psi
+ C_D(\alpha^*) \sin \psi \right),  \\
&
+\left( \frac{r^2}{ \bar v^{*}} - \frac{r^3}{ \bar v^{*2}}  \right) 
\left[ C_L(\alpha^*) \cos \psi - C_D(\alpha^*) \sin \psi \right]
+
\mathcal{O}\left( r^4 \right).
\end{split}
\end{equation}
Note the Taylor series up to 3rd order in $\psi$ for $\cos$ and $\sin$ is,
\begin{equation*}
    \cos \psi =    1 - \tfrac{1}{2}\psi^2 + \mathcal{O}(\psi^4), \qquad
    \sin \psi = \psi - \tfrac{1}{6}\psi^3 + \mathcal{O}(\psi^5).
\end{equation*}
Plugging in all the Taylor series expansions, we get, up through 3rd order in  $\psi$ and $r$,
\begin{equation}
\begin{split}
\psi^{\prime} & = 
 \bar v^* 
\left( -C_L - \psi C_L^{\prime} 
- \tfrac{1}{2}\psi^2 C_L^{\prime \prime}
- \tfrac{1}{6}\psi^3 C_D^{\prime \prime \prime}
+ C_L - \tfrac{1}{2}\psi^2 C_L
- \psi C_D  
+ \tfrac{1}{6}\psi^3 C_D \right), \\
&
+  r \left( -C_L - \psi C_L^{\prime}
- \tfrac{1}{2}\psi^2 C_L^{\prime \prime}
- C_L + \tfrac{1}{2}\psi^2 C_L
+ C_D\psi \right), \\
&
+\left( \frac{r^2}{ \bar v^{*}} - \frac{r^3}{ \bar v^{*2}}  \right) 
\left[ C_L - C_D \psi \right]
+
\mathcal{O}\left( 4 \right),
\end{split}
\end{equation}
where it should be understood that the lift and drag coefficients and all their derivatives (w.r.t.\ angle of attack) are evaluated at the critical point $\alpha^*$, 
and where $\mathcal{O}(4)$ stands for terms which are fourth order or higher in the variables $\psi$ and $r$.

Grouping terms by powers in $\psi$ and $r$, we get
\begin{equation}
\begin{split}
\psi^{\prime} & = 
 \bar v^* \left[- C_L^{\prime} - C_D \right] \psi  
+  2 \left[ - C_L \right] r \\
&
+ \tfrac{\bar v^*}{2}  \left[-  C_L^{\prime \prime}
-  C_L \right]  \psi^2
+ \left[ - C_L^{\prime} +  C_D \right] \psi r 
+ \tfrac{ 1}{\bar v^{*}}[C_L] r^2,
 \\
&
+\tfrac{\bar v^*}{6}[C_D - C_L^{\prime \prime \prime}] \psi^3
+\tfrac{1}{2}[C_L - C_L^{\prime \prime}]\psi^2 r
+\tfrac{ 1}{\bar v^{*}}[-C_D] \psi r^2
+\tfrac{ 1}{\bar v^{*2}}[-C_L] r^3
+ \mathcal{O}(4).
\end{split}
\end{equation}
There are terms linear in $\psi$ and $r$,  terms second-order in $\psi$ and $r$, and 
terms third-order in $\psi$ and $r$.

We can follow a similar procedure for the $r^{\prime}$ expression.
Using the $\sin$ addition formula,
\begin{equation}
\sin(\gamma^* + \psi) = \sin \psi \cos \gamma^* + \cos \psi \sin \gamma^*,
\end{equation}
along with (\ref{Omega_zero_eqs_polar_equil}), we get
\begin{equation}
\begin{split}
 \sin(\gamma^* + \psi) & = 
\left[ 
 \bar v^{*2} C_L \sin \psi + 
 \bar v^{*2} C_D \cos \psi 
\right], \\
& = 
 \bar v^{*2} 
\left[ C_L \sin \psi + C_D \cos \psi 
\right],
\\
& = 
 \bar v^{*2} 
\left[ C_L  \psi - \tfrac{1}{6} C_L \psi^3 + C_D - \tfrac{1}{2} \psi^2 C_D 
 + \mathcal{O}(4).
\right]
\end{split}
\end{equation}
Also,
\begin{equation}
\begin{split}
-  (\bar v^* + r)^2 C_D(\alpha^*+ \psi) &=
-  \bar v^{*2} 
\left[  C_D + \psi C_D^{\prime}
+\tfrac{1}{2}\psi^2 C_D^{\prime \prime} 
+\tfrac{1}{6}\psi^3 C_D^{\prime \prime \prime} +
\mathcal{O} ( 4 ) \right], \\
&
- 2 \bar v^* r
\left[  C_D + \psi C_D^{\prime}
+\tfrac{1}{2}\psi^2 C_D^{\prime \prime} + 
\mathcal{O} ( 3 ) \right], \\
&
-  r^2 
\left[  C_D + \psi C_D^{\prime}
+
\mathcal{O} ( 2 ) \right],
\end{split}
\end{equation}
so
we get
\begin{equation}
\begin{split}
 r^{\prime} & = 
 \bar v^{*2} \left[- C_D^{\prime} + C_L \right] \psi   
+  2 \left[ - \bar v^* C_D \right] r , \\
&
+\tfrac{\bar v^{*2}}{2}  \left[- C_D - C_D^{\prime \prime} \right] \psi^2
+ 2 \bar v^*[-C_D^{\prime}]  \psi r 
+  [-C_D] r^2 , \\
& 
+\tfrac{\bar v^{*2}}{6}[-C_L - C_D^{\prime \prime \prime}] \psi^3
+       \bar v^* [ - C_D^{\prime \prime}]\psi^2 r
+\tfrac{ 1}{2}[-C_D^{\prime}] \psi r^2
+ [ 0] r^3
+ \mathcal{O}(4).
\end{split}
\end{equation}

Putting the $(\psi,r)$ system into matrix form, we have,
\begin{equation}
    \label{eqn:jacobian_polar_local}
	\begin{bmatrix} 
	\psi^{\prime} \\ 
	r^{\prime} 
	\end{bmatrix} =
\underbrace{	\begin{bmatrix}
	\bar v^* \left[- C_L^{\prime}- C_D \right] & 
	\left[ -2 C_L \right] \\ 
	\bar v^{*2} \left[- C_D^{\prime} + C_L \right]&     \left[ -2 \bar v^* C_D \right]
	\end{bmatrix}}_{\mathbf{A}}
		\begin{bmatrix} 
		\psi \\ 
		r 
		\end{bmatrix}
		+ 
		\mathbf{F}(\psi,r) 
		+ \mathcal{O}(4),
\end{equation}
where $\mathbf{F}(\psi,r)$ stands for second and third-order terms, and is given by
\begin{equation}
    \label{eqn:jacobian_polar_local_F}
\mathbf{F}(\psi,r) = 
		\begin{bmatrix} 
		F^1(\psi,r) \\ 
		F^2(\psi,r) 
		\end{bmatrix},
\end{equation}
where
\begin{equation}
 	\label{eqn:jacobian_polar_local_F1}
\begin{split}
F^1 (\psi,r) =& 
 \tfrac{\bar v^*}{2}  \left[
 -  C_L  -  C_L^{\prime \prime} 
 \right]  \psi^2
+ \left[  C_D - C_L^{\prime}  \right] \psi r 
+ \tfrac{ 1}{\bar v^{*}} [C_L] r^2
 \\
+&
\tfrac{\bar v^*}{6}
[ C_D - C_L^{\prime \prime \prime} ] \psi^3
+\tfrac{1}{2} [ C_L - C_L^{\prime \prime} ]\psi^2 r
+\tfrac{ 1}{\bar v^{*}}[ -C_D ] \psi r^2
+\tfrac{ 1}{\bar v^{*2}}[ -C_L ] r^3,
\end{split}
\end{equation}
and,
\begin{equation}
    \label{eqn:jacobian_polar_local_F2}
\begin{split}
F^2(\psi,r) =& 
\tfrac{\bar v^{*2}}{2}  \left[- C_D - C_D^{\prime \prime} \right] \psi^2
+ 2 \bar v^*[-C_D^{\prime}]  \psi r 
+  [-C_D] r^2 \\
+&\tfrac{\bar v^{*2}}{6}[-C_L - C_D^{\prime \prime \prime}] \psi^3
+       \bar v^* [ - C_D^{\prime \prime}]\psi^2 r
+\tfrac{ 1}{2}[-C_D^{\prime}] \psi r^2
+ [ 0] r^3.
\end{split}
\end{equation}

From the $2 \times 2$ (Jacobian) linearization matrix $\mathbf{A}$ in (\ref{eqn:jacobian_polar_local}), 
\begin{equation}\label{linearization_matrix_A}
\mathbf{A}=
\begin{bmatrix}
	\bar v^* \left[- C_L^{\prime}- C_D \right] & 
	\left[ -2 C_L \right] \\ 
	\bar v^{*2} \left[- C_D^{\prime} + C_L \right]&     \left[ -2 \bar v^* C_D \right]
\end{bmatrix},
\end{equation}
we can analytically determine the eigenvalues and eigenvectors in terms of the equilibrium point and characteristics of the lift and drag curves at that point.

For this simple system, we can write the eigenvalue equation in the standard form as
\cite[p.~130]{Strogatz2001}
\begin{equation*}
    \lambda^2 - \bar \tau \lambda + \bar \Delta = 0,
\end{equation*}
where $\bar \tau = \mathrm{trace}(\mathbf{A})$ and $\bar \Delta = \mathrm{det}(\mathbf{A})$. The eigenvalues are
\begin{equation*}
    \lambda_1 = \dfrac{\bar \tau + \sqrt{\bar \tau^2 - 4 \bar \Delta}}{2}, \qquad
    \lambda_2 = \dfrac{\bar \tau - \sqrt{\bar \tau^2 - 4 \bar \Delta}}{2}.
\end{equation*}

The trace of  $\mathbf{A}$ is
\[
\bar \tau = \bar v^* \left[- C_L^{\prime}- 3 C_D \right],
\]
and the determinant of $\mathbf{A}$ is,
\[
\bar \Delta = 2 \bar v^{*2} \left[ C_L^2  + C_D^2 + C_L^{\prime}C_D - C_D^{\prime}C_L \right].  
\]
So,
\begin{equation}
\begin{split}
\bar \tau^2 - 4 \bar \Delta 
& = \bar v^{*2}\left[ (C_L^{\prime} + 3 C_D)^2
-8 (  C_L^2  + C_D^2 + C_L^{\prime}C_D - C_D^{\prime}C_L  ) \right] ,\\
& = \bar v^{*2}\left[ C_L^{\prime 2} + C_D^2 - 2 C_L^{\prime}C_D 
-8 C_L^2 + 8 C_D^{\prime}C_L \right ] ,\\
& = \bar v^{*2}\left[  (C_D - C_L^{\prime})^2 - 8 C_L (C_L - C_D^{\prime}) \right],
\end{split}
\end{equation}
and the eigenvalues are,
\begin{equation*}
    \lambda_{1,2} = \dfrac{\bar \tau \pm \sqrt{\bar \tau^2 - 4 \bar \Delta}}{2}.
\end{equation*}

We can write the eigenvalues more compactly by introducing $\tau$ and $\Delta$,
\begin{equation}
    \begin{split}         
    \tau =&  \left(\tfrac{ C_L^{\prime} }{C_D}\right) + 3,\\
     \Delta =&   \left(\tfrac{C_L}{C_D}\right)^{\prime} + \left(\tfrac{C_L}{C_D}\right)^{2} + 1.
    \end{split}
    \label{tau_and_delta}
\end{equation}
such that,
\begin{equation}
    \begin{split}         
    \bar \tau =& - \frac{ C_D }{( C_L^2 + C_D^2 )^{1/4}}  \tau,
    \\
     \bar  \Delta =& 2 \frac{ C_D^2 }{( C_L^2 + C_D^2 )^{1/2}}  \Delta,
    \end{split}
\end{equation}
in which case,
\begin{equation}
\lambda_{1,2} = \frac{ C_D }{2( C_L^2 + C_D^2 )^{1/4}}
		\left( -\tau \pm \sqrt{\tau^2 - 8 \Delta} \right),
\end{equation}
and since the prefactor,
\begin{equation}
\frac{ C_D }{2( C_L^2 + C_D^2 )^{1/4}},
\end{equation}
is always a positive scalar, the location of the eigenvalues on the complex plane is given solely by $\tau$ and $\Delta$.

\subsection{Hopf bifurcation case}\label{hopf}

We often view the pitch variable $\theta$ as a bifurcation parameter. A Hopf bifurcation occurs when 
$\bar \tau=0$ and $\bar \Delta > 0$, so the eigenvalues are purely imaginary,
\begin{equation*}
    \lambda_{\pm} = \pm i \omega,
\end{equation*}
where,
$\omega = \sqrt{\bar \Delta}>0$.
Suppose this occurs along the branch of equilibria at a particular value of $\theta$ which we will call $\bar \theta$.
By the assumption of $\bar \tau=0$, we conclude that,
\begin{equation}
C_L^{\prime} = -3 C_D,
\end{equation}
and from $\bar \Delta > 0$, we conclude that
\begin{equation}
\begin{split}
C_L > \tfrac{1}{2} \left( C_D^{\prime} + \sqrt{C_D^{\prime 2} + 8 C_D^{2}} \right)
\quad \rm{or} \quad 
C_L < \tfrac{1}{2} \left( C_D^{\prime} - \sqrt{C_D^{\prime 2} + 8 C_D^{2}} \right).
\end{split}
\end{equation}
Notice that the sign of
\begin{equation}
d=\frac{d}{d \theta}\left(\rm{Re}(\lambda(\theta))\right) \bigg|_{\theta=\bar \theta} 
= \tfrac{1}{2}\bar \tau^{\prime} = \tfrac{1}{2}\bar{v}^*( -C_L^{\prime\prime}  -3 C_D^{\prime}),
\end{equation}
is an indication of the type of bifurcation.
If, as $\theta$ increases, the equilibrium point is going from a stable to unstable focus, then $\bar \tau^{\prime}>0$. Otherwise, $\bar \tau^{\prime}<0$. 
Note that,
\begin{equation}
\begin{split}
C_L^{\prime\prime} < -3 C_D^{\prime} &\quad \text{going from stable to unstable}, \bar \tau^{\prime}>0\\
C_L^{\prime\prime} > -3 C_D^{\prime} &\quad \text{going from unstable to stable}, \bar \tau^{\prime}<0 
\end{split}
\end{equation}

For the case of purely imaginary eigenvalues, we have,
\begin{equation}
\mathbf{A}=
\begin{bmatrix}
	\bar v^* 2 C_D  &   -2 C_L  \\ 
	\bar v^{*2} \left(- C_D^{\prime} + C_L \right)&     -\bar v^* 2 C_D 
	\end{bmatrix},
\end{equation}
where the eigenvalues are $\pm i \omega$,
where,
\begin{equation}
\label{Hopf_omega}
\omega = \bar v^{*}  \sqrt{2} \sqrt{C_L^2 - C_L C_D^{\prime} - 2 C_D^2},
\end{equation}
is positive.
We solve for the generalized eigenvectors $\mathbf{u}$ and $\mathbf{v}$,
\begin{equation}
\mathbf{u} = \begin{bmatrix} 2 C_L \\ \bar v^* 2 C_D \end{bmatrix},
\quad 
\mathbf{v} = \begin{bmatrix} 0 \\ \omega \end{bmatrix},
\end{equation}
Define the matrix $\mathbf{P}$ as,
\[
\mathbf{P} = \left[ \mathbf{u} ~~ \mathbf{v} \right],
\]
so 
$\mathbf{u}$ is the first column of $\mathbf{P}$ and
$\mathbf{v}$ is the second column of $\mathbf{P}$.
This matrix defines a linear transformation to the eigenbasis $(x,y)$ via,
\[
		\begin{bmatrix} 
		\psi \\ 
		r 
		\end{bmatrix}
		= 
\mathbf{P}
		\begin{bmatrix} 
		x \\ 
		y 
		\end{bmatrix},
\]
so 
the $x$ coordinate is along the $\mathbf{u}$ direction and
the $y$ coordinate is along the $\mathbf{v}$ direction. Note that
\begin{equation}\label{change_to_xy}
\begin{split}
\psi &= 2 C_L x, \\
r    &= \bar v^* 2 C_D x + \omega y.
\end{split}
\end{equation}
The dynamics in the eigenbasis are,
\begin{equation}
    \label{eqn:eigenbasis_hopf}
	\begin{bmatrix} 
	 x^{\prime} \\ 
	 y^{\prime}
	\end{bmatrix} =
\begin{bmatrix}
	0       & -\omega \\ 
	\omega  &  0  
	\end{bmatrix}
		\begin{bmatrix} 
		x \\ 
		y
		\end{bmatrix}
		+ 
		\mathbf{P}^{-1} 
		\mathbf{F}(2 C_L x, \bar v^* 2 C_D x + \omega y) 
		+ \mathcal{O}(4),
\end{equation}
where $\mathbf{F}$, from \eqref{eqn:jacobian_polar_local_F}, includes the 2nd and 3rd order terms and where,
\begin{equation}
\mathbf{P}^{-1}=
\begin{bmatrix}
	~~\tfrac{1}{2 C_L}  &   0  \\
	-\tfrac{\bar v^{*} C_D}{\omega C_L}  &   \tfrac{1}{\omega} 
	\end{bmatrix}.
\end{equation}
We will re-write the nonlinear terms, defining $\mathbf{f}(x,y)=\mathbf{P}^{-1} \mathbf{F}(2 C_L x, \bar v^* 2 C_D x + \omega y)$, so the resulting equation now has the form,
\begin{equation}
    \label{eqn:eigenbasis_hopf2}
	\begin{bmatrix} 
	 x^{\prime} \\ 
	 y^{\prime}
	\end{bmatrix}  =
\begin{bmatrix}
	0       & -\omega \\ 
	\omega  &  ~~0  
	\end{bmatrix}
		\begin{bmatrix} 
		x \\ 
		y
		\end{bmatrix}
		+ 
		\begin{bmatrix} 
		f^1(x,y) \\ 
		f^2(x,y) 
		\end{bmatrix},
\end{equation}
The coefficient $a$, from \cite{Wiggins2003} and \cite{Marsden1976}, which determines what kind of Hopf bifurcation will occur, can be calculated as,
\begin{equation}
\begin{split}
    \label{eqn:a_hopf}
a = &\tfrac{1}{16}\left[f^1_{xxx} + f^1_{xyy} + f^2_{xxy} + f^2_{yyy}\right]
\\
&+\tfrac{1}{16 \omega} \left[f^1_{xy}(f^1_{xx} + f^1_{yy}) -
f^2_{xy}(f^2_{xx} + f^2_{yy}) 
- f^1_{xx} f^2_{xx} + f^1_{yy} f^2_{yy}\right],
\end{split}
\end{equation}
where all partial derivatives are evaluated at the bifurcation point, $\theta = \bar \theta$, $x=0$, $y=0$,
\begin{equation}
\begin{split}
F^1_{xx} &= ~~\barv ( 4\cl^3 - 4\cl^3 \clpp + 36\cl\cd^2),
\\
F^1_{xy} &= ~~\om 12\cl\cd,
\\
F^1_{yy} &= ~~\barv 4\cl   (\cl^2 - \cl\cdp - 2\cd^2),
\\
F^2_{xx} &= -\barvtwo (4\cl^2\cd +8\cd^3 + 4\cl^2\cdpp +16\cl\cd\cdp ),
\\
F^2_{xy} &= -\barv \om(\cl\clp + \cd^2),
\\
F^2_{yy} &= -\barvtwo 4\cd (\cl^2 - \cl\cdp - 2\cd^2),
\\
F^1_{xxx} &=-\barv(8\cl^3\clppp +24\cl^2\cd\clpp +96 \cl\cd^3-32\cl^3\cd  ),
\\ 
F^1_{xyy} &= -\barv 32 \cd\cl  (\cl^2 - \cl\cdp - 2\cd^2),
\\
F^1_{xxy} &=\barv(4[\cl-\clpp]\cl^2(\tfrac{\om}{\barv})-40\cl\cd^2(\tfrac{\om}{\barv})),
\\ 
F^1_{yyy} &= -\barv 12 \cl  (\cl^2 - \cl\cdp - 2\cd^2)(\tfrac{\om}{\barv}),
\\ 
F^2_{xxy} &= -\barv \om (8\cl^2\cdpp +8\cl\cd\cdp),
\\
F^2_{yyy} &=~~0,
\end{split}
\end{equation}
and we get the partial derivatives of $\mathbf{f}(x,y)$ from the relationship,
\begin{equation}
    \mathbf{f}(x,y)=\mathbf{P}^{-1} \mathbf{F}(x,y)
\end{equation}
which give us,
\begin{equation}
\begin{split}
f^1(x,y) &= ~~\tfrac{1}{2 C_L} F_1 (x,y),
\\
f^2(x,y) &=-\tfrac{\bar v^{*} C_D}{\omega C_L} F_1(x,y) +   \tfrac{1}{\omega} F_2(x,y).
\end{split}
\end{equation}
Knowing the sign of $a$ along with the sign of $\tau^{\prime}$ will determine which of the four cases of Hopf bifurcation is present, via the Poincar\'e-Andronov-Hopf Bifurcation Theorem \cite{Wiggins2003}. 

We also predict that when a limit cycle exists, it will have a period of approximately $T=\tfrac{2 \pi}{\omega}$ where $\omega$ is given from \eqref{Hopf_omega}, and that the radius of the limit cycle in the $(x,y)$ plane, close to the pitch value $\bar \theta$, is given by,
\begin{equation}
\rho =\sqrt{-\tfrac{d}{a} (\theta - \bar \theta)}.
\end{equation}
Notice that the dependence of $\rho$ on the constants $a$ and $d$, as well as distance away from the bifurcation point, $(\theta - \bar \theta)$, reveal how `quickly' the size of the limit cycle grows.
The amplitude of the limit cycle in terms of glide angle $\gamma$ is provided from \eqref{change_to_xy} as, 
\begin{equation}
\rho_{\gamma} = 2 C_L \rho = 2 C_L \sqrt{-\tfrac{d}{a} (\theta - \bar \theta)}.
\end{equation}

\subsection{Stable node case}\label{stable_node}

If $\bar \tau<0$ and $\bar \tau^2 - 4 \bar \Delta > 0$ (so 
$\sqrt{\bar \tau^2 - 4 \bar \Delta}>0$), then we have two real, and negative, eigenvalues.
The larger magnitude eigenvalue is
\begin{equation}
    \lambda_{ss} = \dfrac{\bar \tau - \sqrt{\bar \tau^2 - 4 \bar \Delta}}{2}
    = \tfrac{1}{2} \bar v^* \left(- C_L^{\prime}- 3 C_D 
    - \sqrt{(C_D - C_L^{\prime})^2 - 8 C_L (C_L - C_D^{\prime})} \right),
\end{equation}
and the smaller magnitude eigenvalue is
\begin{equation}\label{smaller_stable}
    \lambda_s = \dfrac{\bar \tau + \sqrt{\bar \tau^2 - 4 \bar \Delta}}{2}
        = \tfrac{1}{2} \bar v^* \left(- C_L^{\prime}- 3 C_D 
    + \sqrt{(C_D - C_L^{\prime})^2 - 8 C_L (C_L - C_D^{\prime})} \right), 
\end{equation}
so $\lambda_{ss} < \lambda_s < 0$,
where `$s$' denotes {\it  stable}  and `$ss$' denotes {\it super stable}. 
Let the corresponding eigenvectors be $\ee_{ss}$ and $\ee_s$, respectively, understood as column vectors. 

Now $\bar \tau<0$ implies that 
\begin{equation}
C_L^{\prime} > - 3 C_D,
\end{equation}
and $\bar \tau^2 - 4 \bar \Delta > 0$ implies that
\begin{equation}
(C_D - C_L^{\prime})^2 > 8 C_L (C_L - C_D^{\prime})
\end{equation}

We can solve for $\ee_s$, since it will give us a local approximation of the {\it terminal velocity manifold (TVM)} described in the text. Suppose all we want is the slope $\bar m$ (in $(\psi,r)$ coordinates), so  we let $\ee_s = [-1, -\bar m]^T$.
From the eigenvector formula,
\begin{equation}
\mathbf{A} \ee_s = \lambda_s \ee_s,
\end{equation}
where,
\begin{equation}
\mathbf{A}=
\begin{bmatrix}
	a  &   b  \\
	c  &  d
	\end{bmatrix},
\end{equation}
we have
\begin{equation}
\bar m = \frac{\lambda_s-a}{b},
\end{equation}
and using
\eqref{linearization_matrix_A} and
\eqref{smaller_stable}, we get
\[
a = \bar v^* \left[- C_L^{\prime}- C_D \right], \quad  
b =\left[ -2 C_L \right] ,
\]
and thus,
\begin{equation}\label{slope_psi_r}
\bar m = \frac{\bar v^*}{4 C_L} \left( C_D -  C_L^{\prime} - \sqrt{(C_D - C_L^{\prime})^2 - 8 C_L (C_L - C_D^{\prime})} \right).
\end{equation} 
We want the slope $m$ in $(\bar v_x,\bar v_z)$ coordinates, so, using the relationship between the cartesian and polar coordinates,
\begin{equation}
\begin{split}
    \bar{v}_x =& ~~~\bar{v} \cos{\gamma}, \\
    \bar{v}_z =& -\bar{v} \sin{\gamma}, 
\end{split}
\end{equation} 
we write the transformation between local vectors,
\begin{equation}
	\begin{bmatrix} 
	d \bar v_x \\ 
	d \bar v_z 
	\end{bmatrix} =
\begin{bmatrix}
	-\bar v^* \sin \gamma^*   &   ~~\cos \gamma^* \\ 
	-\bar v^* \cos \gamma^*  &  - \sin \gamma^*  
	\end{bmatrix}
		\begin{bmatrix} 
		d \psi \\ 
		d r
		\end{bmatrix},
\end{equation}
and letting $dr = \bar m~d \psi$, we get the slope of the terminal velocity manifold,
\begin{equation}\label{slope_v_x_v_z}
 m = \frac{d \bar v_z}{d \bar v_x} = \frac{\bar v^*\cos \gamma^*  + \bar m \sin \gamma^*}{\bar v^*\sin \gamma^* - \bar m \cos \gamma^* },
\end{equation} 
with $\bar m$ as in \eqref{slope_psi_r}. Note, this is the {\it local} slope of the terminal velocity manifold, as evaluated at the stable node point. The slope may change, i.e., the manifold may be curved, as explored in the next case.

For completeness, we also compute the eigenvector $\ee_{ss}= [-1, -\bar n]^T$, and get,
\begin{equation}\label{slope_psi_r_e_ss}
\bar n = \frac{\bar v^*}{4 C_L} \left( C_D -  C_L^{\prime} + \sqrt{(C_D - C_L^{\prime})^2 - 8 C_L (C_L - C_D^{\prime})} \right).
\end{equation}

\subsection{Saddle case}\label{saddle}

If $\bar \Delta < 0$, so $\bar \Delta = - |\bar \Delta|$, then
$\sqrt{\bar \tau^2 - 4 \bar \Delta}=\sqrt{\bar \tau^2 + 4 |\bar \Delta|}>|\bar \tau|$, then we have two real eigenvalues, one negative ($\lambda_s$) and one positive ($\lambda_u$).
The negative eigenvalue is
\begin{equation}
    \lambda_{s} = \dfrac{\bar \tau - \sqrt{\bar \tau^2 - 4 \bar \Delta}}{2}
    = \tfrac{1}{2} \bar v^* \left(- C_L^{\prime}- 3 C_D 
    - \sqrt{(C_D - C_L^{\prime})^2 - 8 C_L (C_L - C_D^{\prime})} \right),
\end{equation}
and the positive eigenvalue is,
\begin{equation}\label{smaller_stable}
    \lambda_u = \dfrac{\bar \tau + \sqrt{\bar \tau^2 - 4 \bar \Delta}}{2}
        = \tfrac{1}{2} \bar v^* \left(- C_L^{\prime}- 3 C_D 
    + \sqrt{(C_D - C_L^{\prime})^2 - 8 C_L (C_L - C_D^{\prime})} \right), 
\end{equation}
Let the corresponding eigenvectors be $\ee_{s}$ and $\ee_u$, respectively, understood as column vectors.

We can solve for $\ee_u$, since it will give us a local approximation of the {\it terminal velocity manifold} described in the text. All we want is the slope $\bar m$ (in $(\psi,r)$ coordinates), so  we let $\ee_u = [-1, -\bar m]^T$.
From the eigenvector formula
\begin{equation}
\mathbf{A} \ee_u = \lambda_u \ee_u,
\end{equation}
where
\begin{equation}
\mathbf{A}=
\begin{bmatrix}
	a  &   b  \\
	c  &  d
	\end{bmatrix},
\end{equation}
we have
\begin{equation}
\bar m = \frac{\lambda_u-a}{b},
\end{equation}
and using
\eqref{linearization_matrix_A} and
\eqref{smaller_stable}, we get
\[
a = \bar v^* \left[- C_L^{\prime}- C_D \right], \quad  
b =\left[ -2 C_L \right] ,
\]
and thus,
\begin{equation}\label{slope_psi_r}
\bar m = \frac{\bar v^*}{4 C_L} \left( C_D -  C_L^{\prime} - \sqrt{(C_D - C_L^{\prime})^2 - 8 C_L (C_L - C_D^{\prime})} \right).
\end{equation} 
We want the slope $m$ in $(\bar v_x,\bar v_z)$ coordinates, so, using the relationship between the cartesian and polar coordinates,
\begin{equation}\label{transformation_v_gamma_to_vx_vz}
\begin{split}
    \bar{v}_x =& ~~~\bar{v} \cos{\gamma}, \\
    \bar{v}_z =& -\bar{v} \sin{\gamma} ,
\end{split}
\end{equation} 
we write the transformation between local vectors,
\begin{equation}
	\begin{bmatrix} 
	d \bar v_x \\ 
	d \bar v_z 
	\end{bmatrix} =
\begin{bmatrix}
	-\bar v^* \sin \gamma^*   &   ~~\cos \gamma^* \\ 
	-\bar v^* \cos \gamma^*  &  - \sin \gamma^*  
	\end{bmatrix}
		\begin{bmatrix} 
		d \psi \\ 
		d r
		\end{bmatrix},
\end{equation}
and letting $dr = \bar m~d \psi$, we get the slope of the terminal velocity manifold,
\begin{equation}\label{slope_v_x_v_z}
 m = \frac{d \bar v_z}{d \bar v_x} = \frac{\bar v^*\cos \gamma^*  + \bar m \sin \gamma^*}{\bar v^*\sin \gamma^* - \bar m \cos \gamma^* }, 
\end{equation} 
with $\bar m$ as in \eqref{slope_psi_r}.
Again, this is the {\it local} slope of the terminal velocity manifold, as evaluated at the saddle point, and may be different from the local slope of the terminal velocity manifold as evaluated at the stable node, if the manifold is curved.

For completeness, we also compute the eigenvector $\ee_{s}= [-1, -\bar n]^T$, and get,
\begin{equation}\label{slope_psi_r_e_ss}
\bar n = \frac{\bar v^*}{4 C_L} \left( C_D -  C_L^{\prime} + \sqrt{(C_D - C_L^{\prime})^2 - 8 C_L (C_L - C_D^{\prime})} \right).
\end{equation}

\subsection{Higher order approximation of terminal velocity manifold}\label{higher_order}
Define the matrix $\mathbf{P}$ as
\begin{align*}
\mathbf{P} &= \left[ \ee_u ~~ \ee_{s} \right], \\
&= \begin{bmatrix}
	-1       	& -1 		\\
	-\bar m  	& -\bar n  
	\end{bmatrix},
\end{align*}
so 
$\ee_u$ is the first column of $\mathbf{P}$ and
$\ee_{s}$ is the second column of $\mathbf{P}$.

This matrix defines a linear transformation to the eigenbasis $(x,y)$ via
\[
		\begin{bmatrix} 
		\psi \\ 
		r 
		\end{bmatrix}
		= 
\mathbf{P}
		\begin{bmatrix} 
		x \\ 
		y 
		\end{bmatrix},
\]
so 
the $x$ coordinate is along the $\ee_{u}$ direction and
the $y$ coordinate is along the $\ee_{s}$ direction.
Note that
\begin{equation}\label{transformation_x_y_to_psi_r}
\begin{split}
\psi &= - x - y, \\
r     &= 	-\bar m x 	 -\bar n y ,
\end{split}
\end{equation}
and
\begin{equation}
\mathbf{P}^{-1}=
\frac{1}{\bar m - \bar n} 
\begin{bmatrix}
	~~\bar n &   -1  \\
	  -\bar m &~~1
	\end{bmatrix}.
\end{equation}

Considering (\ref{eqn:jacobian_polar_local}), we have
\begin{equation}
    \label{eqn:jacobian_polar_local_xy1}
	\begin{bmatrix} 
	x^{\prime} \\ 
	y^{\prime} 
	\end{bmatrix} =
\underbrace{	\mathbf{P}^{-1} \mathbf{A} \mathbf{P} }_{\Lam}
		\begin{bmatrix} 
		x \\ 
		y
		\end{bmatrix}
		+ 
		\mathbf{P}^{-1} \mathbf{F}(x,y) ,
\end{equation}
where $\Lam$ is the diagonalized matrix,
\[
\Lam = \begin{bmatrix}
	\lambda_{u} 	&	 0		 \\
	0           		&  \lambda_{s}  
	\end{bmatrix},
\]
and where 
care must be taken to calculate the second-order terms, $\mathbf{P}^{-1} \mathbf{F}(x,y)$, in terms of $x$ and $y$, where $\mathbf{F}(x,y)$ 
is given as in \eqref{eqn:jacobian_polar_local_F}-\eqref{eqn:jacobian_polar_local_F2}.

We will re-write the nonlinear terms, defining $\mathbf{f}(x,y)=\mathbf{P}^{-1} \mathbf{F}(-x - y, -\bar m x 	 -\bar n y )$, so
 \begin{align*}
\mathbf{f}(x,y) 
&=
\frac{1}{\bar m - \bar n} \begin{bmatrix}
	~~\bar n &   -1  \\
	  -\bar m &~~1
	\end{bmatrix}
 \begin{bmatrix}
	a_1 x^2 + a_2 xy + a_3 y^2\\
	b_1 x^2 + b_2 xy + b_3 y^2
	\end{bmatrix}
	+ \mathcal{O}(3),
\end{align*}
where
 \begin{align*}
a_1 & = a^1 + b^1 \bar m + c^1 \bar m^2 , \\
a_2 & = 2 a^1 + b^1 (\bar m + \bar n) + 2c^1 \bar m \bar n , \\
a_3 & = a^1 + b^1 \bar n + c^1 \bar n^2 , \\
b_1 & = a^2 + b^2 \bar m + c^2 \bar m^2 , \\
b_2 & = 2 a^2 + b^2 (\bar m + \bar n) + 2c^2 \bar m \bar n , \\
b_3 & = a^2 + b^2 \bar n + c^2 \bar n^2  ,\\
a^1 &= \tfrac{\bar v^*}{2}  \left[
 -  C_L  -  C_L^{\prime \prime} 
 \right] ,\\
 b^1 &= \left[  C_D - C_L^{\prime}  \right], \\
 c^1 &= \tfrac{ 1}{\bar v^{*}} [C_L] ,\\
a^2 &= \tfrac{\bar v^{*2}}{2}  \left[- C_D - C_D^{\prime \prime} \right], \\
b^2 &= 2 \bar v^*[-C_D^{\prime}] , \\
c^2 &=  [-C_D].
\end{align*}
We will refer to the components of $\mathbf{f}$ as $(f,g)$.

The resulting equation now has the form,
\begin{equation}
    \label{eqn:eigenbasis_hopf21}
	\begin{bmatrix} 
	 x^{\prime} \\ 
	 y^{\prime}
	\end{bmatrix}  =
\begin{bmatrix}
	\lambda_{u} 	&	 0		 \\
	0           		&  \lambda_{s}  
	\end{bmatrix}
		\begin{bmatrix} 
		x \\ 
		y
		\end{bmatrix}
		+ 
		\begin{bmatrix} 
		f(x,y) \\ 
		g(x,y) 
		\end{bmatrix} ,
\end{equation}
where,
\begin{equation}
\begin{split}
f(x,y) &= c_1 x^2 + c_2 xy + c_3 y^2 + \mathcal{O}(3) , \\
g(x,y) &= d_1 x^2 + d_2 xy + d_3 y^2 + \mathcal{O}(3),
\end{split}
\end{equation}
where,
 \begin{align}
c_i & = \tfrac{1}{\bar m - \bar n} ( ~~\bar n a_i - b_i),  \\
d_i & = \tfrac{1}{\bar m - \bar n} ( -\bar m a_i + b_i).  
\end{align}

We will end up with the expansion about the equilibrium in a form where we can now calculate the terminal velocity manifold. We re-write \eqref{eqn:eigenbasis_hopf21} as,
\begin{equation}\label{manifold_ready}
\begin{split}
x^{\prime} &= \lambda_u    x + f(x,y),  \\
y^{\prime} &= \lambda_{s} y + g(x,y), 
\end{split}
\end{equation}
where $f(x,y)$ is second-order and higher in $x$ and $y$, as is $g(x,y)$.

We assume the terminal velocity manifold is given by $y=H(x)$, where $H(x)$ has the Taylor series expansion form,
\begin{equation}\label{h_function}
H(x)=ax^2 + bx^3 + \mathcal{O}(x^4),
\end{equation}
We can solve for the coefficients $a$ and $b$ by taking the time derivative of $y=H(x)$, which gives
\[
\pa{H}{x} x^{\prime} - y^{\prime} = 0,
\]
i.e.,
\[
\pa{H}{x} 
\left[ \lambda_u    x + f(x,H(x))  \right]
- 
\left[ \lambda_{s} H(x) + g(x,H(x))  \right] = 0,
\]
and equating like powers of $x$,
\[
(2 a x + 3b x^2 + \mathcal{O}(x^3))
\left[ \lambda_u x + c_1 x^2 + \mathcal{O}(x^3) \right]
- 
\left[ \lambda_{s} a x^2 + d_1 x^2  + \mathcal{O}(x^3) \right] = 0,
\]
i.e.,
\[
\left[ a (2 \lambda_u -  \lambda_{s}) -d_1 \right] x^2  = 0,
\]
so
\[
a = \frac{d_1}{(2 \lambda_u -  \lambda_{s})}.
\]
Thus, to a second-order approximation in the $(x,y)$ coordinates, the terminal velocity manifold is expressed as,
\[
y=H(x)=\frac{d_1}{(2 \lambda_u -  \lambda_{s})} x^2 + \mathcal{O}(x^3),
\]
thus, in general the manifold will be curved.
To get the curvature up through third-order terms, we need $b$, so we would have to have $\mathbf{f}(x,y)$ calculated up to the third-order terms.
We note that this whole process can be automated using automatic power series expansion tools \cite{GoKoLoMaMaRo2004}.

To get the terminal velocity manifold in the original $(\bar{v}_x,\bar{v}_z)$ coordinates, we use  \eqref{transformation_x_y_to_psi_r}, 
\eqref{transformation_v_gamma_to_psi_r}, and \eqref{transformation_v_gamma_to_vx_vz}, to get a parametric curve, 
\begin{equation}\label{manifold_vx_vz}
\begin{split}
\bar{v}_x (u) &= ~~(\bar{v}^{*} - \bar m u - \bar n H(u)) \cos ( \gamma^* - u - H(u) ), \\ 
\bar{v}_z (u) &=    -(\bar{v}^{*} - \bar m u - \bar n H(u)) \sin ( \gamma^*  - u - H(u) ),
\end{split}
\end{equation}
parametrized by a curvilinear coordinate $u$ which we take to be in some interval $I \subset \mathbb{R}$, where the function $H$ is as in \eqref{h_function}.

We can determine the lowest order non-linear approximation of the vector field {\it along the one-dimensional terminal velocity manifold}, as,
\begin{equation}\label{manifold_flow2}
\begin{split}
u^{\prime} 
&= \lambda_u    u + f(u,H(u)), \\ 
&= \lambda_u    u + c_1 u^2 + \mathcal{O}(u^3),
\end{split}
\end{equation}
where we are using $u$ as a curvilinear (arc-length) coordinate along the terminal velocity manifold. 
This is the analytical formula for the  `speed' (actually, acceleration) along the terminal velocity curve vs.\ location along that curve.
This tells us that a second equilibrium point (stable) will show up along the terminal velocity manifold at $u=-\lambda_u/ c_1$, which is an approximation of where the stable node is located.

Note  that the form of $u^{\prime}$ as a function of $u$ is topologically conjugate to the $h$ function defined in \eqref{eq:DifferenceEquation}.

It is interesting that the local approximation of the dynamics around the saddle point can imply the existence of the stable point, in agreement with the global result of Theorem \ref{odd_points}. 
Also noteworthy is the fact that the terminal velocity manifold constructed from the saddle point to the stable node is a heteroclinic trajectory (backward asymptotic to the saddle point and forward asymptotic to the stable node; see Figure \ref{fig:AttractingSubmanifold}) along which the relative speed varies according to \eqref{manifold_flow2}.

To find out what role the shape of the terminal velocity manifold plays in modifying the vector field along it, we must consider third-order terms in (\ref{eqn:jacobian_polar_local}), which would give us,  
\begin{equation}\label{manifold_flow3}
\begin{split}
u^{\prime} 
&= \lambda_u    u + f_2(u,H(u)) + f_3(u,H(u)) + \mathcal{O}(u^4), \\
&= \lambda_u    u + c_1 u^2 + c_2 a u^3 + k_1 u^3  + \mathcal{O}(u^4), \\
&= \lambda_u    u + c_1 u^2 + \left[ c_2 \tfrac{d}{(2 \lambda_u -  \lambda_{s})} + k_1 \right] u^3  + \mathcal{O}(u^4),
\end{split}
\end{equation}
where
$f_2(x,y) = c_1 x^2 + c_2 xy + c_3 y^2$ and
$f_3(x,y) = k_1 x^3 + k_2 x^2y + k_3 x y^2 + k_4 y^3$
are the second and third order terms in the $x^{\prime}$ equation of (\ref{manifold_ready}), respectively. 

Note that 
\[
k_1 = \tfrac{1}{\bar m - \bar n} ( \bar n \tilde a_1- \tilde b_1),  
\]
where
\begin{align*}
\tilde a_1 & = - ( A_1 + A_2 \bar m + A_3 \bar m^2   + A_4 \bar m^3 ), \\
\tilde b_1 & = - ( B_1 + B_2 \bar m + B_3 \bar m^2  + B_4 \bar m^3  ),
\end{align*}
and where the $A_i$ and $B_i$ come from the third-order coefficients in \eqref{eqn:jacobian_polar_local_F1} and \eqref{eqn:jacobian_polar_local_F2}, respectively,
\begin{align*}
A_1 & = \tfrac{\bar v^*}{6} [ C_D - C_L^{\prime \prime \prime} ], \\
A_2 & = \tfrac{1}{2} [ C_L - C_L^{\prime \prime} ], \\
A_3 & = \tfrac{ 1}{\bar v^{*}}[ -C_D ],  \\
A_4 & = \tfrac{ 1}{\bar v^{*2}}[ -C_L ], \\
B_1 & = \tfrac{\bar v^{*2}}{6}[-C_L - C_D^{\prime \prime \prime}], \\
B_2 & = \bar v^* [ - C_D^{\prime \prime}], \\
B_3 & = \tfrac{ 1}{2}[-C_D^{\prime}], \\
B_4 & = 0. 
\end{align*}
We note that the third-order coefficient $b$ is given by
\[
b = \frac{g_1 - a(2c_1 - d_2)}{3 \lambda_u -  \lambda_{s}},
\]
where,
\[
g_1 = \tfrac{1}{\bar m - \bar n} ( -\bar m \tilde a_1+ \tilde b_1).
\]

\end{appendices}

\end{document}